\documentclass[12pt,a4paper]{amsart}

\usepackage{euscript,amsfonts,amssymb,amsmath,amscd,amsthm,enumerate,hyperref,mathrsfs}

\usepackage{tikz,bm,float}
\usetikzlibrary{calc}
\colorlet{lgray}{white!85!black}
\colorlet{lred}{white!75!red}

\usepackage{bbm}

\usepackage{comment}

\usepackage{graphicx}

\usepackage{color}

\usepackage[margin=1.1in]{geometry}
\newtheorem{theorem}{Theorem} 
\newtheorem*{theorem*}{Theorem}
\newtheorem{lemma}[theorem]{Lemma}
\newtheorem{definition}[theorem]{Definition}
\newtheorem{proposition}[theorem]{Proposition}

\theoremstyle{remark}
\newtheorem{remark}[theorem]{Remark}
\newtheorem{example}[theorem]{Example}

\numberwithin{equation}{section} \numberwithin{theorem}{section}

\newcommand{\N}{\mathbb N}

\newcommand{\Z}{\mathbb Z}

\newcommand{\R}{\mathbb R}

\usepackage{MnSymbol}

\usepackage{skak}

\usepackage{adjustbox}
\usetikzlibrary{arrows}
\usetikzlibrary{decorations.pathreplacing}

\begin{document}

\title{Mallows Product Measure}

\author[Alexey Bufetov]{Alexey Bufetov}\address{Institute of Mathematics, Leipzig University, Augustusplatz 10, 04109 Leipzig, Germany.}\email{alexey.bufetov@gmail.com}

\author[Kailun Chen]{Kailun Chen}\address{Institute of Mathematics, Leipzig University, Augustusplatz 10, 04109 Leipzig, Germany.}\email{kailunxsx@gmail.com}

\begin{abstract}
Q-exchangeable ergodic distributions on the infinite symmetric group were classified by Gnedin-Olshanski (2012). In this paper, we study a specific linear combination of the ergodic measures and call it the Mallows product measure. From a particle system perspective, the Mallows product measure is a reversible stationary blocking measure of the infinite-species ASEP and it is a natural multi-species extension of the Bernoulli product blocking measures of the one-species ASEP. Moreover, the Mallows product measure can be viewed as the universal product blocking measure of interacting particle systems coming from random walks on Hecke algebras. 

For the random infinite permutation distributed according to the Mallows product measure we have computed the joint distribution of its neighboring displacements, as well as several other observables. The key feature of the obtained formulas is their remarkably simple product structure. We project these formulas to ASEP with finitely many species, which in particular recovers a recent result of Adams-Balázs-Jay, and also to ASEP(q,M).  

Our main tools are results of Gnedin-Olshanski about ergodic Mallows measures and shift-invariance symmetries of the stochastic colored six vertex model discovered by Borodin-Gorin-Wheeler and Galashin.
\end{abstract}

\maketitle


\section{Introduction}
Let $S_{n}$ be a symmetric group. The Mallows measure on $S_{n}$ assigns to a permutation $\sigma \in S_{n}$ the probability proportional to $q^{\operatorname{inv}(\sigma)}$, where $\operatorname{inv}(\sigma)$ is the number of inversions in $\sigma$, and $q \in[0 ; 1)$. This measure appeared in various contexts, including ranking problems in statistics (see Mallows \cite{Mal57}), random sorting algorithms and trees (see Benjamini-Berger-Hoffman-Mossel \cite{BBHM05}, Diaconis-Ram \cite{DR00}, Evans-Grübel-Wakolbinger \cite{EGW12}), and multi-species interacting particle systems (see e.g. Bufetov \cite{Buf20}, Bufetov-Nejjar \cite{BN22}, Zhang \cite{Zha22}, He-Schmidt \cite{HS23}).

Gnedin-Olshanski \cite{GO10, GO12} have extended this definition to the case of infinite permutations of integers $\mathbb{Z} \rightarrow \mathbb{Z}$, which captures the key property of the Mallows measures in the finite case: q-exchangeability (see Definition \ref{def:q-exchangeability} below). They classified all ergodic q-exchangeable measures on bijections $\mathbb{Z} \rightarrow \mathbb{Z}$; it turns out that there is a one-parameter family of such measures $\mathcal{M}_{c}, c \in \mathbb{Z}$ (see Theorem \ref{thm:M_0} and Proposition \ref{cor:convex-mixture}).

In this paper, we define and study a one-parameter family $\mathcal{M}^{p}_{\alpha}$ of linear combinations of the measures $\mathcal{M}_{c}$ (see the definition in \eqref{def:Product Mallows} below; $\alpha$ is a real valued parameter, while "$p$" stands for "product"). Clearly, these measures are also q-exchangeable, but no longer ergodic. The main feature of these measures is that some of their observables have a surprisingly simple product form. 

Our first main result is the joint distribution of any $k$ neighboring displacements.
\begin{theorem}[Theorem \ref{prop:neighbor}]
\label{thm:neighbor}
Let $x_{1} \leq x_{2} \leq \cdots \leq x_{k}$ be integers, $k \in \N$. Let $\omega$ be the random permutation of $\mathbb{Z}$ distributed according to the Mallows product measure $\mathcal{M}^{p}_{\alpha}$. For any fixed integer $i \in \mathbb{Z}$, the joint distribution of $k$ neighboring displacements $D_{i+j}=\omega(i+j)-(i+j)$ for $j=1,2,\cdots,k$, is given by 
\begin{align}
\label{intro:neighbor}
\mathbb{P} \left( D_{i+1} =x_{1}, D_{i+2} =x_{2}, \cdots, D_{i+k}=x_{k} \right) =\frac{(1-q)^{k}\alpha^{k}q^{\sum_{j=1}^{k}x_{j}-\frac{k^2}{2}}}{\prod_{j=1}^{k}(1+\alpha q^{x_j+2j-k-\frac{3}{2}})(1+\alpha q^{x_j+2j-k-\frac{1}{2}})}.
\end{align}
\end{theorem} 
The product formula in the right-hand side of \eqref{intro:neighbor} is remarkably simple. For ergodic measures on infinite permutations, analogous observables have significantly more complicated structure found by Gnedin-Olshanski (see Theorem \ref{thm:distribution-Mallows} below). For finite permutations distributed according to the finite Mallows measure, to our knowledge, no explicit formulas for such observables are available.

As explained in Remark \ref{rmk:Pk} below, the condition $x_{1} \leq x_{2} \leq \cdots \leq x_{k}$ in Theorem \ref{thm:neighbor} is not restrictive: Theorem \ref{thm:neighbor} combined with the q-exchangeability property provide the complete joint distribution of images of arbitrarily many neighboring integers. Nevertheless, it does not provide a simple answer to the question regarding the joint distribution of images of arbitrary, not necessarily neighboring, integers. While such a distribution is a marginal distribution of \eqref{intro:neighbor}, the involved summations are highly nontrivial. Our next main result provides a partial answer to this question. 

\begin{theorem}[Theorem \ref{prop:arbitrary}]
\label{thm:arbitrary}
Let $x_{1}>x_{2}>\cdots>x_{k}$, $i_1 < i_2 < \dots < i_k$ be integers, $k \in \mathbb{N}$. Let $\omega$ be the random permutation of $\mathbb{Z}$ distributed according to the Mallows product measure $\mathcal{M}^{p}_{\alpha}$. We have
\begin{equation}
\label{intro:separate}
\mathbb{P} \left(  \omega(i_1) =x_{1}, \omega(i_2)=x_{2}, \cdots,  \omega(i_k)=x_{k} \right)
=\prod_{j=1}^{k}\frac{(1-q) \alpha q^{x_{j}-i_j-\frac{1}{2}}}{(1+\alpha q^{x_{j}-i_j-\frac{1}{2}})(1+\alpha q^{x_{j}-i_j+\frac{1}{2}})}.
\end{equation}
\end{theorem}

Theorem \ref{thm:arbitrary} provides a product formula for the joint distribution of $ \omega(i_1)$, $\omega(i_2)$,....  $\omega(i_k)$ in the region $x_{1}>x_{2}>\cdots>x_{k}$ only. The constraint is significant here: We expect that the joint distribution is much more complicated in other regions, see Remark \ref{rmk:constrait}.

We also establish a product formula for a slightly different observable (it illustrates directly that the Mallows product measure is the natural multi-species extension of the Bernoulli product blocking measures, see Definition \ref{prop:ASEP blocking} and Remark \ref{rmk:ext-Ber})
\begin{theorem}[Theorem \ref{prop:correlation}]
\label{thm:correlation}
Let $x_1 \geq x_2 \geq \cdots \geq x_k$, $i_1 < i_2 < \dots < i_k $ be integers, and let $\omega$ be the random permutation distributed according to the Mallows product measure $\mathcal{M}^{p}_{\alpha}$. We have
\begin{align}
\label{intro:correlation}
\mathbb{P} \left( \omega(i_1) \leq x_1, \omega(i_2) \leq x_2, \cdots, \omega(i_k) \leq x_k \right) 
=\prod_{j=1}^{k} \frac{1}{1+\alpha q^{x_j-i_j+\frac{1}{2}}}.
\end{align}
\end{theorem}

Interestingly enough, we are currently unable to prove Theorems \ref{thm:arbitrary} and \ref{thm:correlation} via summing expressions from Theorem \ref{thm:neighbor} (see Section \ref{ssec:two-separate}). Instead, we use in the proof nontrivial symmetries of the stochastic colored six vertex model established by Borodin-Gorin-Wheeler \cite{BGW22} and Galashin \cite{Gal21}, and combine them with Theorem \ref{thm:neighbor}. 


The Mallows product measure can be interpreted as a stationary reversible product blocking measure of a multi-species asymmetric simple exclusion process (=ASEP), and the formulas above can be interpreted as the distribution of second-, third-, ... class particles under such a measure. We provide several interpretations and corollaries of our results in Section \ref{sec:particle-process} and Section \ref{sec:asymptotics}. In particular, in Theorem \ref{thm:d-second} we recover a recent result of Adams-Balázs-Jay \cite{ABJ23} regarding the distribution of several second class particles in the ASEP blocking measure.

From a more general perspective, as noted in \cite{Buf20}, the Mallows measure is a reversible stationary measure for a class of multi-species interacting particle systems (such as the stochastic six-vertex model, q-TASEP, ASEP(q,M), etc) that can be realized as random walks on Hecke algebras. Therefore, the Mallows product measure is the universal reversible product blocking measure for all these systems, and formulas above can be used for finding properties of second-, third-,.. class particles in these systems. In this direction, we provide an example how our formulas can be also interpreted as the distribution of the second class particle in the ASEP(q,M) product blocking measure, see Section \ref{sec:ASEP-qm}. 

The paper is organized as follows. In Section \ref{sec:prelim} we recall the results of \cite{GO12} and known results regarding the one-species ASEP blocking measures. In Section \ref{sec:Mal-Prod-Meas} we define the Mallows product measure and prove Theorem \ref{thm:neighbor}. We also provide a partial computation towards Theorems \ref{thm:arbitrary} and \ref{thm:correlation} via the direct approach in that Section. In Section \ref{sec:Symmetries} we recall the setup and the main result of \cite{Gal21}, and use it for the derivation of Theorems \ref{thm:arbitrary} and \ref{thm:correlation} from Theorem \ref{thm:neighbor}. In Section \ref{sec:particle-process} we provide several applications of our results to interacting particle systems, and in Section \ref{sec:asymptotics} we record a couple of corollaries regarding the asymptotics of these formulas in the $q \to 1$ regime. 

\subsection*{Acknowledgments} 
We are very grateful to M. Balazs for many helpful discussions regarding blocking measures of ASEP and for sharing with us an early version of \cite{ABJ23}. A part of our motivation was to understand their result from the Mallows measure perspective. We are very grateful to an anonymous referee for helpful comments. 
Both authors were partially supported by the European Research Council (ERC), Grant Agreement No. 101041499.

\section{Preliminaries}
\label{sec:prelim}

In this section, we recall stationary blocking (=reversible) measures of the one-species ASEP and of the infinite-species ASEP on $\mathbb{Z}$. The ergodic stationary blocking measures of infinite-species ASEP on $\mathbb{Z}$ are given by the Mallows measures introduced in \cite{GO12}.

\subsection{One-species ASEP and stationary blocking measure}

First, let us recall the definition of the one-species ASEP model. Throughout the paper, we use a parameter $q \in [0,1)$.

\begin{definition}
The one-species asymmetric simple exclusion process (=one-species ASEP) is a continuous time Markov process with states $\eta=\{\eta_{x}\}_{x\in\mathbb{Z}}$ from the state space $\{0,1\}^{\mathbb{Z}}$. The random variables $\eta_{x}$ are called occupation variables and can be thought of as the indicator function of the event that a particle is present at site $x$. The dynamics of this process is as follows: for each pair of neighboring sites $(y,y+1)$, the following exchanges happen in continuous time
\begin{align*}
\eta \mapsto \eta^{y,y+1} \quad \text{at rate 1} \quad \text{if} \ (\eta_y, \eta_{y+1})=(1,0),\\
\eta \mapsto \eta^{y,y+1} \quad \text{at rate q} \quad \text{if} \ (\eta_y, \eta_{y+1})=(0,1),
\end{align*}
where $\eta^{y,y+1}$ denotes the state in which the value of the occupation variables at site $y$ and $y+1$ are swapped and all other variables remain the same as in $\eta$. All exchanges occur independently of each other according to exponential clocks of rate 1. In other words, ASEP is defined by its generator $L$ which acts on local functions $f:\{0,1\}^{\mathbb{Z}}\to\mathbb{Z}$ by
\begin{align*}
(Lf)(\eta)=\sum_{y\in\mathbb{Z}}[\eta_{y}(1-\eta_{y+1})+q(1-\eta_y)\eta_{y+1}](f(\eta^{y,y+1})-f(\eta)).
\end{align*}
The existence of a Markov process with this generator is well-known, see e.g. \cite{Lig05}.
\end{definition}

Next, we recall the stationary measures of the one-species ASEP model on $\Z$. Liggett \cite{Lig76} has classified all ergodic stationary measures. They form two families:  Translation-invariant i.i.d. Bernoulli measures and blocking measures. The former are non-reversible, while the latter are reversible. In this paper, we focus on reversible stationary blocking measures. 

The common feature of blocking measures is that with probability 1 a random configuration consists of only holes to the left of a (random) point, and consists of only particles to the right of (another) random point. The one-species ASEP model has two natural families of stationary blocking measures: product measures and ergodic measures.

\begin{definition}
\label{prop:ASEP blocking}

The \textit{product} blocking measure on the set of configurations $\eta \in \{0,1\}^{\mathbb{Z}}$ is defined by the condition that the occupation variables $\eta_i$ are independent, Bernoulli distributed and 
\begin{align}
\label{eq:product Bernoulli measure}
\mathbb{P} \left( \eta_i=1 \right) = \frac{1}{1+\alpha^{-1} q^{i-\frac12}}, \qquad \mathbb{P} \left( \eta_i=0 \right) = \frac{\alpha^{-1} q^{i-\frac12}}{1+\alpha^{-1} q^{i-\frac12}}, \qquad i \in \Z,
\end{align}
where $\alpha \in \mathbb{R}_{>0}$ is a parameter. We denote this measure by $\mu_{\alpha}^p$ (where "p" stands for "product"). 
\end{definition} 

To define ergodic blocking measures, we introduce $\{0,1\}^{\mathbb{Z}} \rightarrow \mathbb{Z}_{+} \cup\{\infty\}$ functions
\begin{align*}
N_{\mathrm{p}}(\eta):=\sum_{i=-\infty}^{0} \eta_{i}  \quad \text { and } \quad N_{\mathrm{h}}(\eta):=\sum_{i=1}^{\infty}\left(1-\eta_{i}\right) .
\end{align*}
Namely $N_{\mathrm{p}}$ counts the number of particles $\left(\eta_{i}=1\right)$ to the left of position $\frac{1}{2}$, and $N_{\mathrm{h}}$ counts the number of holes $\left(\eta_{i}=0\right)$ to the right of position $\frac{1}{2}$. Therefore, the relevant state space of configurations for blocking measures can be written as 
\begin{align*}
\Omega:=\left\{\eta \in \{0,1\}^{\Z}: N_{\mathrm{p}}(\eta)<\infty  \text { and }  N_{\mathrm{h}}(\eta)<\infty \right\} .
\end{align*}
Next, we set
\begin{align*}
N(\eta):=N_{\mathrm{h}}(\eta)-N_{\mathrm{p}}(\eta), \qquad N(\eta): \Omega \rightarrow \mathbb{Z},
\end{align*}
and note that this quantity is conserved by the ASEP dynamics. We also define the set 
\begin{align*}
\Omega^{n}:=\{\eta \in \Omega: N(\eta)=n\},
\end{align*}
for all $n \in \mathbb{Z}$. 

\begin{definition}
Let $\eta \in \Omega^{0}$. We say that a pair $i<j$, $i,j \in \Z$, is an $\eta$-inversion, if $\eta_i=1$ and $\eta_j=0$. Let $I(\eta)$ be the number of $\eta$-inversions. The ergodic blocking measure on $\Omega^{0}$ is defined via assigning to $\eta$ the probability 
$$
\mu_{(0)} ( \eta) := q^{I(\eta)} Z, \qquad \eta \in \Omega^{0}, 
$$
where $Z$ is a normalizing factor (it does not depend on $\eta$).
All other ergodic blocking measures are obtained from $\mu_{(0)}$ by a shift; therefore, for each $n \in \Z$ there is exactly one ergodic reversible stationary measure of ASEP supported on $\Omega^{n}$ (we denote it by $\mu_{(n)}$), and this exhausts the list of ergodic reversible stationary measures of ASEP on $\Z$.
\end{definition}

The Bernoulli product measures from Definition \ref{prop:ASEP blocking} are not ergodic, thus, they can be decomposed into a linear combination of ergodic ones. 
For this, let us also recall the Jacobi triple product formula:
\begin{align*}
\sum_{c \in \Z} \frac{ \alpha^c q^{c^2/2}}{ (q;q)_{\infty} \prod_{k=0}^{\infty} \left( 1+\alpha q^{k+1/2} \right) \prod_{k=0}^{\infty} \left( 1+\alpha^{-1} q^{k+1/2} \right) } = 1, \qquad \alpha \in \R_{>0}, \ q \in [0,1).
\end{align*}

\begin{proposition}
One has the following decomposition
\begin{equation}
\label{eq:ergodic-decomposition}
\mu_{\alpha}^p = \sum_{c \in \Z} \frac{ \alpha^c q^{c^2/2}}{ (q;q)_{\infty} \prod_{k=0}^{\infty} \left( 1+\alpha q^{k+1/2} \right) \prod_{k=0}^{\infty} \left( 1+\alpha^{-1} q^{k+1/2} \right) } \mu_{(c)}.
\end{equation}
\end{proposition} 

Up to our knowledge, the earliest reference where this statement was directly presented is \cite{Bor07}. See also \cite{BB18, BB19, ABJ23} for more detail. 

\subsection{Infinity-species ASEP and Mallows measure}
\label{ssec:infty-ASEP}

In this section, we have a similar discussion as in the previous one, but for the infinite-species ASEP. In particular, we recall the results of Gnedin-Olshanski \cite{GO12} regarding the Mallows measure.  

Let us start with introducing the ASEP with infinitely many types of particles.

\begin{definition}[infinity-species ASEP]
We define the model on $\mathbb{Z}$. The system is a continuous-time Markov chain with the state-space $\mathbb{Z}^{\mathbb{Z}}$. For a configuration $\eta=\{\eta_i, i\in \mathbb{Z}\}$, we say that $\eta_i$ is the type of the particle at site $i$. In this paper, we mostly consider configurations $\eta=\{\eta_i, i\in \mathbb{Z}\}$ coming from infinite permutations $\eta: \mathbb{Z} \to \mathbb{Z}$.

The dynamics of the process is as follows. If $\eta_i>\eta_{i+1}$, then the values $\eta_{i}$ and $\eta_{i+1}$ are exchanged at rate 1. If instead $\eta_{i}<\eta_{i+1}$, then the values $\eta_{i}$ and $\eta_{i+1}$ are exchanged at rate $q$. That is, neighboring particles rearrange themselves into increasing order at rate 1, and into decreasing order at rate q. We refer to e.g. \cite{AAV11}, \cite{AHR09} for the more detailed definition of the infinity-species ASEP and its correctness.
\end{definition}

\begin{remark}
\label{rmk:projection}
If we map all the types $\eta_i>0$ to the same type and call it a "particle"; and we map all the types $\eta_i \leq 0$ to the same type and call it a "hole", then we get the one-species ASEP from infinity-species ASEP. Similarly, one can also consider the multi-species ASEP model with finitely many species as the degeneration of the infinity-species ASEP. 
\end{remark}

The set of infinite permutations $\eta: \mathbb{Z} \to \mathbb{Z}$ is uncountable, therefore, unlike the previous section, one is unable to define the probability measures of interest via probabilities of each individual configuration. Instead, the probability measures are described via standard cylinder probabilities. 

Let us recall the notion of q-exchangeability.

\begin{definition}[q-exchangeability]
\label{def:q-exchangeability}
A q-exchangeable probability measure on the set of infinite permutations $\omega: \mathbb{Z} \to \mathbb{Z}$ is a measure which satisfies the following condition: If both values of $\omega(i)$ and $\omega(i+1)$ are specified by a cylinder set, and these values of $\omega(i)$ and $\omega(i+1)$ are swapped, then the probability of a cylinder is multiplied by $q^{sign(\omega(i+1)-\omega(i))}$.
\end{definition}

It is easy to observe that the detailed balance equation for the infinite-species ASEP is equivalent to the q-exchangeability property. Thus, the problem of finding stationary reversible measures of infinite-species ASEP on permutations $\Z \to \Z$ is equivalent to that of finding q-exchangeable probability measures on the same set. The ergodic q-exchangeable measures on permutations $\Z \to \Z$ were classified by Gnedin and Olshanski in \cite{GO12}. They called the obtained measures \textit{Mallows measures}. Let us recall their construction. 

For this, we need more notations from \cite{GO12}. For $a \in \mathbb{Z}$ we set $\mathbb{Z}_{<a}:=$ $\{i \in \mathbb{Z}: i<a\}$. Likewise, we define the subsets $\mathbb{Z}_{\leq a}, \mathbb{Z}_{>a}$, and $\mathbb{Z}_{\geq a}$. Also, set
\begin{align*}
\mathbb{Z}_{-}:=\mathbb{Z}_{\leq 0}=\{\ldots,-1,0\}, \quad \mathbb{Z}_{+}:=\mathbb{Z}_{\geq 1}=\{1,2, \ldots\} \text {. }
\end{align*}
Let $\mathfrak{S}$ denote the group of all permutations of $\mathbb{Z}$. We associate with $\sigma \in \mathfrak{S}$ an infinite 0-1 matrix $A=A(\sigma)$ of format $\mathbb{Z} \times \mathbb{Z}$ such that the $(i, j)$-th entry of $A(\sigma)$ is $\mathbf{1}_{(\sigma(j)=i)}$.  Write $A=A(\sigma)$ as a $2 \times 2$ block matrix
\begin{align*}
A=\left[\begin{array}{ll}
A_{--} & A_{-+} \\
A_{+-} & A_{++}
\end{array}\right]
\end{align*}
according to the splitting $\mathbb{Z}=\mathbb{Z}_{-} \sqcup \mathbb{Z}_{+}$. For matrix $B$ let $\operatorname{rk}(B)$ be the rank of $B$. 

\begin{definition}
We call $\sigma$ admissible if both $\operatorname{rk}\left(A_{-+}\right)$and $\operatorname{rk}\left(A_{+-}\right)$are finite. The set of admissible permutations will be denoted $\mathfrak{S}^{\text {adm }} \subset \mathfrak{S}$. For $\sigma \in \mathfrak{S}^{\text {adm }}$ we define the balance as
\begin{align*}
b(\sigma):=\operatorname{rk}\left(A_{-+}\right)-\operatorname{rk}\left(A_{+-}\right) .
\end{align*}
%
We call a permutation balanced if it is admissible and has balance 0. We shall denote the set of balanced permutations $\mathfrak{S}^{\text {bal }}$.
\end{definition}

\begin{theorem}\cite[Theorem 3.3]{GO12}
\label{thm:M_0}
There exists a unique probability measure $\mathcal{M}_0$ on $\mathfrak{S}^{\text{bal}} \subset \mathfrak{S}$ which is q-exchangeable.
\end{theorem}

The measure $\mathcal{M}_0$ is referred to as a Mallows measure. Let us mention the following important property.
\begin{proposition}\cite[Corollary 3.4]{GO12}
The measure $\mathcal{M}_0$ is invariant under the inversion map $\sigma \to \sigma^{-1}$.
\end{proposition}

We use the notation $s^{c}$ for the shift permutation: $i \mapsto i-c, c \in \mathbb{Z}$.  Applying the shifts $\sigma \to s^{c}\sigma$, one obtains from $ \mathcal{M}_0$ a family $\{\mathcal{M}_{c}\}_{c \in \mathbb{Z}}$ of q-exchangeable measures with pairwise disjoint supports, which live on the larger subgroup $\mathfrak{S}^{\text {adm}}$. The support of $\mathcal{M}_{c}$ is the set of all permutations with balance $c$.

The family $\{\mathcal{M}_{c}\}_{c \in \mathbb{Z}}$ exhausts the list of all ergodic  q-exchangeable measures on $\mathfrak{S}$: 
\begin{proposition}\cite[Corollary 3.6]{GO12}
\label{cor:convex-mixture}
Each q-exchangeable probability measure on $\mathfrak{S}$ is a unique convex mixture of the measures $\mathcal{M}_{c}$ over $c\in\mathbb{Z}$.
\end{proposition}

Let us also mention explicitly the translation-invariance property: 
\begin{proposition}\cite[Section 5]{GO12}
\label{prop:stationary}
Fix any $c \in \mathbb{Z}$, let $\Sigma^{c}$ be the random permutation of $\mathbb{Z}$ with law $\mathcal{M}_{c}$. Consider the two-sided infinite random sequence of displacements
\begin{align*}
D_{i}^{c}:=\Sigma^{c}(i)-i, \quad i \in \mathbb{Z} .
\end{align*}
The sequence $\left\{D_{i}^{c}\right\}_{i \in \mathbb{Z}}$ is stationary in $i \in \mathbb{Z}$. 
\end{proposition}

The properties of the distribution of $\left\{D_{i}^{c}\right\}_{i \in \mathbb{Z}}$ are of central interest for the study of the Mallows measure. Let us recall the result of \cite{GO12} regarding it. 
Below we use a standard notation for $q$-Pochhammer symbols:
\begin{align*}
(q ; q)_{n}=\prod_{k=1}^{n}\left(1-q^{k}\right), \quad (q ; q)_{\infty}=\prod_{k=1}^{\infty}\left(1-q^{k}\right).
\end{align*}
%
\begin{theorem} \cite[Theorem 5.1 and Theorem 6.1]{GO12}
\label{thm:distribution-Mallows}
For any fixed $i \in \mathbb{Z}$, the distribution of a displacement $D_{i}^{c}$ is given by
\begin{align}
\label{eq:one-dimen}
\mathbb{P}(D^{c}_{i}=d)=(1-q)(q ; q)_{\infty} \sum_{\{r, \ell \geq 0\ :\ r-\ell=d-c\}} \frac{q^{r \ell+r+\ell}}{(q ; q)_{r}(q ; q)_{\ell}}, \quad d \in \mathbb{Z} .
\end{align}
For $k=1,2, \ldots$ and integers $d_{1} \leq \cdots \leq d_{k}$, the joint distribution of displacements $D_{1}^{c}, \dots, D_{k}^{c}$ are given by
\begin{multline}
\label{eq:finite-dimen}
\mathbb{P}\left(D^{c}_{1}=d_{1}, \ldots, D^{c}_{k}=d_{k}\right) \\
=(1-q)^{k} q^{-k(k+1) / 2} (q ; q)_{\infty} \prod_{m=2}^{k}  (q ; q)_{d_{m}-d_{m-1}} \sum \frac{q^{\sum_{1 \leq i \leq j \leq k}\left(b_{i}+1\right)\left(a_{j}+1\right)}}{(q ; q)_{b_1} \ldots (q ; q)_{b_k} (q ; q)_{a_1} \ldots(q ; q)_{a_k}},
\end{multline}
where the summation is over all nonnegative integers $a_{1}, b_{1}, \ldots, a_{k}, b_{k}$ which satisfy the constraints
\begin{align*}
\left(b_{1}+\cdots+b_{m}\right)-\left(a_{m}+\cdots+a_{k}\right)=d_{m}-c, \quad m=1, \ldots, k .
\end{align*}
\end{theorem}

\begin{remark}
The formula \eqref{eq:finite-dimen} provides the complete distribution of finitely many neighboring displacements: The condition $d_{1} \leq \cdots \leq d_{k}$ does not lead to the loss of generality due to the q-exchangeability property (see Remark \ref{rmk:Pk} below). 
\end{remark}

In the end of this section, we prove that the Mallows measure $\mathcal{M}_0$ appears as a limit of the infinite-species ASEP started from a certain initial condition. In detail, we consider the infinity-species ASEP started from the initial configuration $\pi(i)=i$, for all $i \in \mathbb{Z}$. Let us denote the (random) configuration of this process at time $t \in \R_{\ge 0}$ as $\pi^{step}_t$. In particular, $\pi^{step}_0$ is the identity permutation of $\Z$. 

\begin{proposition}
\label{prop:converg-to-stat}

For any $N \in \N$ one has

\begin{equation}
\label{eq:converg-to-stat}
\left( \pi^{step}_t (1), \pi^{step}_t (2), \dots, \pi^{step}_t (N) \right) \xrightarrow[t \to \infty]{} \left( w_0 (1), w_0 (2), \dots, w_0 (N) \right),
\end{equation}
where $w_0$ is the random permutation distributed according to $\mathcal M_0$, and the convergence is in distribution. 

\end{proposition}

\begin{proof}

Let us project the infinite-species ASEP $\pi^{step}_t$ to the ASEP with finitely many species by combining all $\ge(N+1)$-type particles into a single type (let us say, $N+1$), and all $\le 0$-type particles 
into a single type (let us say, $0$). After such a projection, the resulting ASEP has the same distribution of positions of types of particles $1,2,3, ... , N$ as before (since these types were not affected by the projection), and the state space of the resulting ASEP is countable (while the space of all permutations $\Z \to \Z$ is uncountable). Therefore, by the standard ergodic theorem for Markov chains with countably many states, we have the convergence of the process to its unique stationary distribution, which implies the convergence to the right-hand side of \eqref{eq:converg-to-stat}.

\end{proof}

\section{Mallows Product Measure}
\label{sec:Mal-Prod-Meas}

It is fair to say that the simplest infinite-ASEP stationary blocking measures are not ergodic measures, but rather certain mixtures of them. This is very visible at the level of the one-species case (see Definition \ref{prop:ASEP blocking}). Therefore, it is reasonable to ask what happens at the level of infinite permutations. 

In detail, we start with the ergodic Mallows measure $\mathcal{M}_0$ and its shifts $\mathcal{M}_{c}$, $c \in \Z$. Let us define a new probability measure on $\mathfrak{S}$ using the same coefficients as in \eqref{eq:ergodic-decomposition}:
\begin{equation}
\label{def:Product Mallows}
\mathcal{M}^{p}_{\alpha} := \sum_{c \in \Z} \frac{ \alpha^c q^{c^2/2}}{ (q;q)_{\infty} \prod_{k=0}^{\infty} \left( 1+\alpha q^{k+1/2} \right) \prod_{k=0}^{\infty} \left( 1+\alpha^{-1} q^{k+1/2} \right) } \mathcal{M}_{c}, \qquad \alpha \in \mathbb{R}_{>0}.
\end{equation}
We call this measure the \textit{Mallows product measure} with parameter $\alpha$ ($p$ in the notation $\mathcal{M}^{p}_{\alpha}$ stands for "product"). This is the central object of the paper. Due to the properties of ergodic Mallows measures recalled in Section \ref{ssec:infty-ASEP}, we conclude that the Mallows product measure \eqref{def:Product Mallows} has analogous properties:

\begin{enumerate}
\item q-exchangeablity: the Mallows product measure $\mathcal{M}^{p}_{\alpha}$ is a q-exchangeable probability measure on $\mathfrak{S}$.

\item translation-invariance: For any random permutation $\omega$ of $\mathbb{Z}$ distributed according to the Mallows product measure $\mathcal{M}^{p}_{\alpha}$, the joint distribution of the neighboring displacements $D_i=\omega(i)-i$ for $i=1,2,\cdots,k$ does not change if we simultaneously shift all indices $1,2,\cdots,k$ by a constant.

\item inversion-invariance: The measure $\mathcal{M}^{p}_{\alpha}$ is invariant under the inversion map $\omega \to \omega^{-1}$.
\end{enumerate}

\subsection{The joint distribution of neighboring displacements}
\label{ssec:neighboring-displacements}

Let $\omega$ be the random permutation of $\mathbb{Z}$ distributed according to $\mathcal{M}^{p}_{\alpha}$. We are first interested in the distribution of $\omega(0)$. It can be obtained by a direct computation from the definition \eqref{def:Product Mallows} and Theorem \ref{thm:distribution-Mallows}. For any $x\in\mathbb{Z}$, we can express the distribution of $\omega(0)$ in the following way:
\begin{align}
\label{eq:w(0)-1}
\mathbb{P} \left( \omega(0) =x \right) = Z \ \sum_{c \in \Z} \alpha^c q^{c^2/2}  \sum_{\{r, \ell \geq 0: r-\ell=x-c\}} \frac{q^{r\ell+r+\ell}}{(q;q)_{\ell} (q;q)_r},
\end{align}
where $Z=(1-q)\prod_{k=0}^{\infty} \left( 1+\alpha q^{k+1/2} \right)^{-1} \prod_{k=0}^{\infty} \left( 1+\alpha^{-1} q^{k+1/2} \right)^{-1}$.  The equation \eqref{eq:w(0)-1} can be further simplified into the following form:
\begin{align}
\label{eq:w(0)-2}
\mathbb{P} \left( \omega(0) =x \right) = Z \sum_{\{r,l:r,l \ge 0\}} \alpha^{l-r+x} q^{(l-r+x)^2/2} \frac{q^{rl+r+l}}{(q;q)_l (q;q)_r},
\end{align}
Canceling out the factor $q^{r \ell}$, we can factorize the double sum \eqref{eq:w(0)-2} into the product form:
\begin{align}
\label{eq:w(0)-3}
\mathbb{P} \left( \omega(0) =x \right) =  Z \alpha^{x} q^{x^2/2} \sum_{r \ge 0} \frac{ \alpha^{-r} q^{r^2/2 -rx +r} }{(q;q)_r} \sum_{l \ge 0} \frac{ \alpha^{l} q^{l^2/2 + lx +l} }{(q;q)_l}.
\end{align}
By using the identity of Euler:
\begin{align}
\label{eq:Euler}
\sum_{n=0}^{\infty}\frac{q^{\frac{n(n-1)}{2}}}{(q;q)_n}z^{n}=\prod_{n=0}^{\infty}(1+q^nz), \quad |q|<1,
\end{align}
and the finite q-Binomial theorem:
\begin{align}
\label{eq:q-Binomial}
\sum_{k=0}^{n}\frac{(q;q)_n}{(q;q)_k (q;q)_{n-k}} q^{\frac{k(k-1)}{2}} x^{k}=\prod_{m=0}^{n-1}(1+q^mx),
\end{align}
we can write \eqref{eq:w(0)-3} as follows:
\begin{equation}
\label{eq:w(0)-4}
\mathbb{P} \left( \omega(0) =x \right) = \frac{ (1-q) \alpha^{x} q^{x^2/2} \prod_{m=0}^{\infty} (1 + \alpha^{-1} q^{-x+3/2 +m}) \prod_{m=0}^{\infty} (1 + \alpha q^{x+3/2 +m}) }{\prod_{k=0}^{\infty} \left( 1+\alpha q^{k+1/2} \right) \prod_{k=0}^{\infty} \left( 1+\alpha^{-1} q^{k+1/2} \right)} .
\end{equation}
Finally, by using Lemma \ref{lem:A.1}, we simplify \eqref{eq:w(0)-4} into
\begin{align}
\label{eq:single}
\mathbb{P} \left( \omega(0) =x \right) = \frac{(1-q) \alpha q^{x-1/2}}{ (1+ \alpha q^{x-1/2}) (1+ \alpha q^{x+1/2})}.
\end{align}
This computation already suggests that the Mallows product measure $\mathcal{M}^{p}_{\alpha}$ is indeed simpler than ergodic ones $\{\mathcal{M}_{c}\}_{c \in \mathbb{Z}}$.

\begin{remark}
\label{rmk:Bernoulli}
Let $\omega$ be the random permutation of $\mathbb{Z}$ distributed according to $\mathcal{M}^{p}_{\alpha}$. By the translation-invariance property of the Mallows product measure, we have
\begin{align}
\label{eq:single-i}
\mathbb{P} \left( \omega(i)-i =x \right) = \frac{(1-q) \alpha q^{x-1/2}}{ (1+ \alpha q^{x-1/2}) (1+ \alpha q^{x+1/2})},
\end{align}
where $i \in \mathbb{Z}$. Then one can derive the following formula:
\begin{multline}
\label{eq:Bernoulli}
\mathbb{P}(\omega(i) \leq 0)=\sum_{x=-\infty}^{-i} \frac{(1-q) \alpha q^{x-1/2}}{ (1+ \alpha q^{x-1/2}) (1+ \alpha q^{x+1/2})}\\
=\sum_{x=-\infty}^{-i}\left(\frac{1}{1+ \alpha q^{x+1/2}}-\frac{1}{1+ \alpha q^{x-1/2}}\right)=\frac{\alpha^{-1}q^{i-\frac12}}{1+\alpha^{-1}q^{i-\frac12}},\ i \in \mathbb{Z}.
\end{multline}
By Remark \ref{rmk:projection}, we can project the infinite-species ASEP into the one-species ASEP. Equation \eqref{eq:Bernoulli} means that the Mallows product measure \eqref{def:Product Mallows} is projected to the Bernoulli product blocking measure from Definition \ref{prop:ASEP blocking}, as expected.
\end{remark}

Following the above ideas, we establish the formula for the finite-dimensional distributions of the Mallows product measure.
\begin{theorem}
\label{prop:neighbor}

Let $x_{1} \leq x_{2} \leq \cdots \leq x_{k}$ be integers, $k \in \Z_{+}$. Let $\omega$ be the random permutation of $\mathbb{Z}$ distributed according to the Mallows product measure $\mathcal{M}^{p}_{\alpha}$. For any fixed integer $i \in \mathbb{Z}$ and any $k$, the joint distribution of $k$ neighboring displacements $D_{i+j}=\omega(i+j)-(i+j)$ for $j=1,2,\cdots,k$ is given by 
\begin{align}
\label{eq:neighbor}
\mathbb{P} \left( D_{i+1} =x_{1}, D_{i+2} =x_{2}, \cdots, D_{i+k}=x_{k} \right) =\frac{(1-q)^{k}\alpha^{k}q^{\sum_{j=1}^{k}x_{j}-\frac{k^2}{2}}}{\prod_{j=1}^{k}(1+\alpha q^{x_j+2j-k-\frac{3}{2}})(1+\alpha q^{x_j+2j-k-\frac{1}{2}})}.
\end{align}
\end{theorem} 

\begin{proof}
By the translation-invariance property of the Mallows product measure, we only need to prove the $i=0$ case. Using Definition \eqref{def:Product Mallows} and Theorem \ref{thm:distribution-Mallows}, we have 
\begin{multline*}
\mathbb{P} \left( D_1 =x_1, D_2 =x_2, \cdots, D_k=x_k \right)\\
=Z \prod_{m=2}^{k}(q;q)_{x_m-x_{m-1}}
\cdot \sum_{c \in \Z} \alpha^{c} q^{c^2/2}
\sum_{a_1,b_1,\cdots,a_k,b_k \in A}\frac{q^{\sum_{1 \leq i \leq j \leq k}\left(b_{i}a_{j}+b_{i}+a_{j}\right)}}{(q ; q)_{b_1} \ldots (q ; q)_{b_k} (q ; q)_{a_1} \ldots(q ; q)_{a_k}},
\end{multline*}
where $Z=(1-q)^{k}\prod_{r=0}^{\infty} \left( 1+\alpha q^{r+1/2} \right)^{-1} \prod_{\ell=0}^{\infty} \left( 1+\alpha^{-1} q^{\ell+1/2} \right)^{-1}$, and $A$ is the admissible set for the nonnegative integers $a_1,b_1,\cdots,a_k,b_k$:
\begin{align*}
\left(b_{1}+\cdots+b_{m}\right)-\left(a_{m}+\cdots+a_{k}\right)=x_{m}-c, \quad c\in \Z, \quad m=1, \ldots, k.
\end{align*}
The above conditions mean that $a_1,\cdots,a_{k-1}$ have finite range: $0 \leq a_i \leq x_{i+1}-x_{i}$ for $i=1,\cdots,k-1$, and $b_2,\cdots,b_k$ can be expressed through $x_1, \cdots, x_k$ and $a_1, \cdots, a_{k-1}$: $b_j=x_j-x_{j-1}-a_{j-1}$  for $j=2,\cdots,k$. Therefore, we have only two free parameters $b_1$ and $a_k$ which are arbitrary nonnegative integers. We denote the set of admissible values for $a_1,\cdots,a_{k-1}$ by $B$ and write 
\begin{multline*}
\mathbb{P} \left( D_1 =x_1, D_2 =x_2, \cdots, D_k=x_k \right)\\
=Z \prod_{m=2}^{k}(q;q)_{x_m-x_{m-1}}
\mathop{\sum}_{{a_1,\cdots,a_{k-1}\in B}\atop a_k,b_1 \geq 0}
\alpha^{(x_1-b_1+a_1+\cdots+a_k)} q^{(x_1-b_1+a_1+\cdots+a_k)^2/2}\\
\times \frac{q^{\sum_{1 \leq j \leq k}\left(b_1 a_{j}+b_1+a_j\right)}q^{\sum_{2 \leq i \leq j \leq k}\left[(x_i-x_{i-1}-a_{i-1})a_{j}+(x_i-x_{i-1}-a_{i-1})+a_j\right]}}{(q ; q)_{b_1}(q ; q)_{x_2-x_1-a_1}\ldots (q ; q)_{x_k-x_{k-1}-a_{k-1}} (q ; q)_{a_1} \ldots (q ; q)_{a_{k-1}}(q ; q)_{a_k}}.
\end{multline*}
Simplifying the above formula we get:
\begin{multline*}
\mathbb{P} \left( D_1 =x_1, D_2 =x_2, \cdots, D_k=x_k \right)\\
=Z \alpha^{x_1} q^{\frac{x_1^2}{2}+x_2+\cdots+x_k-(k-1)x_1}
\prod_{j=1}^{k-1}\sum_{a_{j}=0}^{x_{j+1}-x_{j}}\frac{\alpha^{a_j}(q;q)_{x_{j+1}-x_{j}}q^{\frac{a_j^2}{2}+a_j x_j+(2j-k)a_{j}}}{(q;q)_{a_j}(q;q)_{x_{j+1}-x_j-a_j}}\\
\times \sum_{a_{k} \geq 0}\frac{\alpha^{a_k}q^{\frac{a_k^2}{2}+a_k x_k+ka_{k}}}{(q;q)_{a_k}}
\sum_{b_{1} \geq 0}\frac{\alpha^{-b_1}q^{\frac{b_1^2}{2}-b_1 x_1+kb_{1}}}{(q;q)_{b_1}}.
\end{multline*}
By the identity of Euler \eqref{eq:Euler} and the finite q-Binomial theorem \eqref{eq:q-Binomial}, we can get that:
\begin{multline*}
\mathbb{P} \left( D_1 =x_1, D_2 =x_2, \cdots, D_k=x_k \right)\\
=Z \alpha^{x_1} q^{\frac{x_1^2}{2}+x_2+\cdots+x_k-(k-1)x_1}
\prod_{i=1}^{k-1}\prod_{j=0}^{x_{i+1}-x_i-1}\left(1+\alpha q^{j+x_i+\frac{1}{2}+(2i-k)}\right)\\
\times \prod_{m=0}^{\infty}\left(1+\alpha q^{m+x_k+\frac{1}{2}+k}\right)
\prod_{n=0}^{\infty}\left(1+\alpha^{-1} q^{m-x_1+\frac{1}{2}+k}\right).
\end{multline*}
By using Lemma \ref{lem:A.2}, we obtain \eqref{eq:neighbor}.
\end{proof}

\begin{remark}
In the case $k=1$, the result \eqref{eq:neighbor} agrees with \eqref{eq:single-i}.
\end{remark}

\begin{remark}
\label{rmk:Pk}
By the q-exchangeability of the Mallows product measure, we can get the joint probability distribution of any $k$ neighboring elements of random permutation $\omega(i+1), \omega(i+2), \cdots, \omega(i+k)$ and remove the restriction $x_1 \leq x_2 \leq \cdots \leq x_k$. Let $x_1,x_2,\cdots,x_k$ be any $k$ distinct integers, and assume that they are linearly ordered according to a finite permutation $\sigma \in S_k$: $x_{\sigma(1)} <x_{\sigma(2)}< \cdots < x_{\sigma(k)}$. For the random permutation $\omega$ of $\mathbb{Z}$ distributed according to the Mallows product measure $\mathcal{M}^{p}_{\alpha}$, the joint distribution of $\omega(i+j)$ for $j=1,2,\cdots,k$ can be written in a product form:
\begin{align*}
&\mathbb{P} \left( \omega(i+1) =x_1, \omega(i+2) =x_2, \cdots, \omega(i+k)=x_k\right)\\
=& q^{inv(\sigma)} \mathbb{P} \left( \omega(i+1) =x_{\sigma(1)}, \omega(i+2) =x_{\sigma(2)}, \cdots, \omega(i+k)=x_{\sigma(k)}\right)\\
=& q^{inv(\sigma)} \mathbb{P} \left( D_{i+1} =x_{\sigma(1)}-(i+1), D_{i+2}=x_{\sigma(2)}-(i+2), \cdots, D_{i+k}=x_{\sigma(k)}-(i+k)\right) \\
=& \frac{(1-q)^{k}\alpha^{k}q^{inv(\sigma)}q^{\sum_{j=1}^{k}x_{j}-\frac{k(2k+2i+1)}{2}}}{\prod_{j=1}^{k}(1+\alpha q^{x_{\sigma(j)}-i+j-k-\frac{3}{2}})(1+\alpha q^{x_{\sigma(j)}-i+j-k-\frac{1}{2}})},
\end{align*}
where we used the standard notation $inv(\sigma) :=\#\{ (i,j): 1 \leq i<j \leq k, \sigma(i)>\sigma(j)\}$.
\end{remark}

\begin{remark}
\label{rmk:inverse}
From Remark \ref{rmk:Pk}, we have an explicit formula for the case of the opposite ordering. For 
$x_{1} > x_{2} > \cdots > x_{k}$, one has 
\begin{multline}
\label{eq:inverse}
\mathbb{P} \left(  \omega(i+1) =x_1, \omega(i+2) =x_2, \cdots, \omega(i+k)=x_k \right)\\
=\frac{(1-q)^{k}\alpha^{k}q^{\sum_{j=1}^{k}x_{j}-\frac{k(k+2i+2)}{2}}}{\prod_{j=1}^{k}(1+\alpha q^{x_{j}-i-j-\frac{1}{2}})(1+\alpha q^{x_{j}-i-j+\frac{1}{2}})}=\prod_{j=1}^{k}\frac{(1-q)\alpha q^{x_{j}-i-j-\frac{1}{2}}}{(1+\alpha q^{x_{j}-i-j-\frac{1}{2}})(1+\alpha q^{x_{j}-i-j+\frac{1}{2}})}.
\end{multline}
\end{remark}

\subsection{The joint distribution of two arbitrary displacements}
\label{ssec:two-separate}

Theorem \ref{prop:neighbor} and Remark \ref{rmk:Pk} provide the joint distribution of images of arbitrarily many neighboring elements of the random permutation distributed according to the Mallows product measure. Therefore, any finite-dimensional distribution of them can be obtained as a marginal. Nevertheless, the question regarding the joint distribution of arbitrary, not necessarily neighboring, elements is a complicated one, since the summations involved in the projection are quite non-trivial. In this section we illustrate these difficulties in a particular case of the joint distribution of two elements, which we attempt through a direct approach. A more general Theorem \ref{thm:arbitrary} will be obtained in the next Section via a different route. 

\begin{proposition}
\label{prop:separate}
Fix $i \in \mathbb{Z}$, and $k \in \mathbb{Z}_{\ge 2}$. Let $x_{1}, x_{k}$ be two integers and $x_{1}>x_{k}$. For the random permutation $\omega$ of $\mathbb{Z}$ distributed according to the Mallows product measure $\mathcal{M}^{p}_{\alpha}$, the joint distribution of $\omega(i+1)$ and $\omega(i+k)$ has the following form:
\begin{multline}
\label{eq:separate-2}
\mathbb{P} \left(  \omega(i+1) =x_{1}, \omega(i+k)=x_{k} \right) \\
=\frac{(1-q) \alpha q^{x_{1}-i-\frac{3}{2}}}{(1+\alpha q^{x_{1}-i-\frac{3}{2}})(1+\alpha q^{x_{1}-i-\frac{1}{2}})}\cdot\frac{(1-q) \alpha q^{x_{k}-i-k-\frac{1}{2}}}{(1+\alpha q^{x_{k}-i-k-\frac{1}{2}})(1+\alpha q^{x_{k}-i-k+\frac{1}{2}})}.
\end{multline}
\end{proposition}

Before we give the proof of Proposition \ref{prop:separate}, we need a lemma.
\begin{lemma}
\label{lem:summation}
For arbitrary integers $i,k, c \in \mathbb{Z}$, and $d\in\mathbb{Z}_{\geq 2}$, we have the following equation:
\begin{multline}
\sum_{x_c<x_{c+1}<\cdots<x_{c+d-1}<x_{c+d}}\frac{q^{\sum_{j=c+1}^{c+d-1}}x_j}{\prod_{j=c+1}^{c+d-1}(1+\alpha q^{x_j+j-i-k-\frac32})(1+\alpha q^{x_j+j-i-k-\frac12})}\\
=\sum_{\ell=0}^{d-1} \frac{(-1)^{\ell} \alpha^{1-d} q^{\sum_{j=c+1}^{c+\ell}(k-j+i+\frac32)}q^{(d-1-\ell)(k-c-\ell+i+\frac12)}}{\prod_{j=1}^{\ell}(1-q^{j})\prod_{j=1}^{d-1-\ell}(1-q^{j})\prod_{j=c+1}^{c+\ell}(1+\alpha q^{x_c+j-i-k-\frac12})\prod_{j=c+\ell+1}^{c+d-1}(1+\alpha q^{x_{c+d}+j-i-k-\frac32})}.
\end{multline}
\end{lemma}

\begin{proof}
For any fixed parameter $c$, we only need to use the mathematical induction on the parameter $d$ to prove the lemma.
\end{proof}

\begin{proof}[Proof of Proposition \ref{prop:separate}]
By definition, we have
\begin{multline*}
\mathbb{P} \left(  \omega(i+1) =x_{1}, \omega(i+k)=x_{k} \right) \\
=\sum_{x_{2},\ldots,x_{k-1} \in \mathbb{Z}} \mathbb{P} \left(\omega(i+1) =x_1, \omega(i+2) =x_2, \cdots, \omega(i+k-1)=x_{k-1}, \omega(i+k)=x_k \right).
\end{multline*}
Using the explicit formula from Remark \ref{rmk:Pk} and extracting the factors containing $x_{1}$ and $x_{k}$, we can get 
\begin{multline}
\label{eq:separate0}
\mathbb{P} \left(  \omega(i+1) =x_{1}, \omega(i+k)=x_{k} \right)\\
=\sum_{\sigma \in S_{k}: \ \sigma^{-1}(k)<\sigma^{-1}(1)}\ q^{inv(\sigma)}\frac{(1-q)^{k}\alpha^{k}q^{x_{1}+x_{k}-\frac{k(2k+2i+1)}{2}}}{\prod_{j\in\{1, k\}}(1+\alpha q^{x_j+\sigma^{-1}(j)-i-k-\frac{3}{2}})(1+\alpha q^{x_j+\sigma^{-1}(j)-i-k-\frac{1}{2}})}\\
\times \left(\mathop{\sum}_{x_{2},\ldots,x_{k-1}\in \mathbb{Z}: \atop -\infty <x_{\sigma(1)} <x_{\sigma(2)}< \cdots < x_{\sigma(k)} < \infty}\frac{q^{\sum_{j=2}^{k-1}x_{j}}}{\prod_{j=2}^{k-1}(1+\alpha q^{x_{j}+\sigma^{-1}(j)-i-k-\frac{3}{2}})(1+\alpha q^{x_{j}+\sigma^{-1}(j)-i-k-\frac{1}{2}})}\right),
\end{multline}
where $-\infty < x_{\sigma(1)} <x_{\sigma(2)}< \cdots < x_{\sigma(k)} < \infty$ gives the constraint condition on the variables $x_2,x_3,\cdots,x_{k-1}$. Next, we calculate the explicit formula for 
\begin{align}
\label{eq:separate1}
\mathop{\sum}_{x_{2},\ldots,x_{k-1}\in \mathbb{Z}: \atop -\infty < x_{\sigma(1)} <x_{\sigma(2)}< \cdots < x_{\sigma(k)} < \infty}\frac{q^{\sum_{j=2}^{k-1}x_{j}}}{\prod_{j=2}^{k-1}(1+\alpha q^{x_{j}+\sigma^{-1}(j)-i-k-\frac{3}{2}})(1+\alpha q^{x_{j}+\sigma^{-1}(j)-i-k-\frac{1}{2}})},
\end{align}
where $\sigma \in S_{k}$ and $\sigma^{-1}(k)<\sigma^{-1}(1)$. Let $1 \leq a < b \leq k$, and suppose that $\sigma(a)=k$ and $\sigma(b)=1$, then \eqref{eq:separate1} can be written as the product of three summations:
\begin{align}
\label{eq:separate2}
\sum_{-\infty < x_{\sigma(1)} < \cdots < x_{\sigma(a-1)} < x_{k}}\frac{q^{\sum_{j=1}^{a-1}x_{\sigma(j)}}}{\prod_{j=1}^{a-1}(1+\alpha q^{x_{\sigma(j)}+j-i-k-\frac{3}{2}})(1+\alpha q^{x_{\sigma(j)}+j-i-k-\frac{1}{2}})}, 
\\
\label{eq:separate3}
\sum_{x_{k} < x_{\sigma(a+1)} < \cdots < x_{\sigma(b-1)} < x_{1}}\frac{q^{\sum_{j=a+1}^{b-1}x_{\sigma(j)}}}{\prod_{j=a+1}^{b-1}(1+\alpha q^{x_{\sigma(j)}+j-i-k-\frac{3}{2}})(1+\alpha q^{x_{\sigma(j)}+j-i-k-\frac{1}{2}})}, \\
\label{eq:separate4}
\sum_{x_{1} < x_{\sigma(b+1)} < \cdots < x_{\sigma(k)} < +\infty }\frac{q^{\sum_{j=b+1}^{k}x_{\sigma(j)}}}{\prod_{j=b+1}^{k}(1+\alpha q^{x_{\sigma(j)}+j-i-k-\frac{3}{2}})(1+\alpha q^{x_{\sigma(j)}+j-i-k-\frac{1}{2}})},
\end{align}
Setting $c=0$, $d=a$, $x_c=-\infty$, and $x_{c+d}=x_k$ in Lemma \ref{lem:summation}, we can get the explicit formula of \eqref{eq:separate2}:
\begin{multline}
\label{eq:separate2.1}
\sum_{-\infty < x_{\sigma(1)} < \cdots < x_{\sigma(a-1)} < x_{k}}\frac{q^{\sum_{j=1}^{a-1}x_{\sigma(j)}}}{\prod_{j=1}^{a-1}(1+\alpha q^{x_{\sigma(j)}+j-i-k-\frac{3}{2}})(1+\alpha q^{x_{\sigma(j)}+j-i-k-\frac{1}{2}})} \\
=\frac{\alpha^{-(a-1)}q^{(a-1)(k+i+\frac{1}{2})}}{\prod_{j=1}^{a-1}(1-q^j)\prod_{j=1}^{a-1}(1+\alpha q^{x_{k}+j-i-k-\frac{3}{2}})}.
\end{multline}
Setting $c=a$, $d=b-a$, $x_c=x_k$, and $x_{c+d}=x_1$ in Lemma \ref{lem:summation}, we can get the explicit formula of \eqref{eq:separate3}:
\begin{multline}
\label{eq:separate3.1}
\sum_{x_k < x_{\sigma(a+1)} < \cdots < x_{\sigma(b-1)} < x_1}\frac{q^{\sum_{j=a+1}^{b-1}x_{\sigma(j)}}}{\prod_{j=a+1}^{b-1}(1+\alpha q^{x_{\sigma(j)}+j-i-k-\frac{3}{2}})(1+\alpha q^{x_{\sigma(j)}+j-i-k-\frac{1}{2}})}\\
=\sum_{\ell=0}^{b-a-1} \frac{(-1)^{\ell} \alpha^{-(b-a-1)} q^{\sum_{j=a+1}^{a+\ell}(k-j+i+\frac32)}q^{(b-a-1-\ell)(k-a-\ell+i+\frac12)}}{\prod_{j=1}^{\ell}(1-q^{j})\prod_{j=1}^{b-a-1-\ell}(1-q^{j})\prod_{j=a+1}^{a+\ell}(1+\alpha q^{x_k+j-i-k-\frac12})\prod_{j=a+\ell+1}^{b-1}(1+\alpha q^{x_{1}+j-i-k-\frac32})}.
\end{multline}
Setting $c=b$, $d=k-b+1$, $x_c=x_1$, and $x_{c+d}=+\infty$ in Lemma \ref{lem:summation}, we can get the explicit formula of \eqref{eq:separate4}:
\begin{multline}
\label{eq:separate4.1}
\sum_{x_1 < x_{\sigma(b+1)} < \cdots < x_{\sigma(k)} < +\infty }\frac{q^{\sum_{j=b+1}^{k}x_{\sigma(j)}}}{\prod_{j=b+1}^{k}(1+\alpha q^{x_{\sigma(j)}+j-i-k-\frac{3}{2}})(1+\alpha q^{x_{\sigma(j)}+j-i-k-\frac{1}{2}})} \\
=\sum_{\ell=0}^{k-b} (-1)^{\ell} \alpha^{-(k-b)} \frac{q^{\sum_{j=b+1}^{b+\ell}(k-j+i+\frac32)}q^{(k-b-\ell)(k-b-\ell+i+\frac12)}}{\prod_{j=1}^{\ell}(1-q^{j})\prod_{j=1}^{k-b-\ell}(1-q^{j})\prod_{j=b+1}^{b+\ell}(1+\alpha q^{x_1+j-i-k-\frac12})}\\
=\frac{q^{(k-b)x_1+\frac{(k-b)(k-b+1)}{2}}}{\prod_{j=1}^{k-b}(1-q^{j})\prod_{j=b+1}^{k}(1+\alpha q^{x_1+j-i-k-\frac12})}.
\end{multline}
where we used the finite q-Binomial theorem \eqref{eq:q-Binomial} in the last step (See Lemma \ref{lem:A.3} for the detailed proof).
%
Substituting \eqref{eq:separate2.1}, \eqref{eq:separate3.1} and \eqref{eq:separate4.1} into \eqref{eq:separate0}, and using the fact that
\begin{align}
\label{eq:inv-summation}
\sum_{\sigma \in S_k:\ \sigma(a)=k, \sigma(b)=1 } q^{inv(\sigma)}=q^{k-a+b-2}\prod_{j=1}^{k-2}\frac{1-q^j}{1-q}, 
\ \ 1 \leq a < b \leq k,
\end{align}
we can get the following formula:
\begin{multline*}
\mathbb{P} \left(  \omega(i+1) =x_{1}, \omega(i+k)=x_{k} \right)\\
=\sum_{1 \leq a < b \leq k}\frac{\alpha^{k}(1-q)^{k}q^{x_{1}+x_k-\frac{k(2k+2i+1)}{2}}q^{k-a+b-2}\prod_{j=1}^{k-2}\frac{1-q^j}{1-q}}{(1+\alpha q^{x_k+a-i-k-\frac{3}{2}})(1+\alpha q^{x_k+a-i-k-\frac{1}{2}})(1+\alpha q^{x_1+b-i-k-\frac{3}{2}})(1+\alpha q^{x_1+b-i-k-\frac{1}{2}})}\\
\times \frac{\alpha^{-(a-1)}q^{(a-1)(k+i+\frac{1}{2})}}{\prod_{j=1}^{a-1}(1-q^j)\prod_{j=1}^{a-1}(1+\alpha q^{x_{k}+j-i-k-\frac{3}{2}})}
\cdot \frac{q^{(k-b)x_1+\frac{(k-b)(k-b+1)}{2}}}{\prod_{j=1}^{k-b}(1-q^{j})\prod_{j=b+1}^{k}(1+\alpha q^{x_1+j-i-k-\frac12})}\\
\times\sum_{\ell=0}^{b-a-1} \frac{(-1)^{\ell} \alpha^{-(b-a-1)} q^{\sum_{j=a+1}^{a+\ell}(k-j+i+\frac32)}q^{(b-a-1-\ell)(k-a-\ell+i+\frac12)}}{\prod_{j=1}^{\ell}(1-q^{j})\prod_{j=1}^{b-a-1-\ell}(1-q^{j})\prod_{j=a+1}^{a+\ell}(1+\alpha q^{x_k+j-i-k-\frac12})\prod_{j=a+\ell+1}^{b-1}(1+\alpha q^{x_{1}+j-i-k-\frac32})}.
\end{multline*}
We can find that the right-hand side of \eqref{eq:separate-2} is contained in above formula:
\begin{multline*}
\mathbb{P} \left(  \omega(i+1) =x_{1}, \omega(i+k)=x_{k} \right)\\
=\frac{(1-q)^{2}\alpha^{2}q^{x_1+x_k-2i-k-2}}{(1+\alpha q^{x_k-i-k-\frac{1}{2}})(1+\alpha q^{x_k-i-k+\frac{1}{2}})(1+\alpha q^{x_1-i-\frac{3}{2}})(1+\alpha q^{x_1-i-\frac{1}{2}})}\sum_{1 \leq a < b \leq k}q^{k-a+b-2}\prod_{j=1}^{k-2}\frac{1-q^j}{1-q}\\
\times q^{2(i+1)-\frac{k(2k+2i-1)}{2}} \frac{q^{(a-1)(k+i+\frac12)}}{\prod_{j=1}^{a-1}\frac{1-q^j}{1-q}\prod_{j=1}^{a-1}(1+\alpha q^{x_k+j-i-k+\frac{1}{2}})} \frac{\alpha^{k-b} q^{(k-b)x_1+\frac{(k-b)(k-b+1)}{2}}}{\prod_{j=1}^{k-b}\frac{1-q^j}{1-q}\prod_{j=b+1}^{k}(1+\alpha q^{x_1+j-i-k-\frac{5}{2}})}\\
\times \sum_{\ell=0}^{b-a-1} \frac{(-1)^{\ell} q^{\sum_{j=a+1}^{a+\ell}(k-j+i+\frac32)}q^{(b-a-1-\ell)(k-a-\ell+i+\frac12)}}{\prod_{j=1}^{\ell}\frac{1-q^j}{1-q}\prod_{j=1}^{b-a-1-\ell}\frac{1-q^j}{1-q}\prod_{j=a+1}^{a+\ell}(1+\alpha q^{x_k+j-i-k-\frac12})\prod_{j=a+\ell+1}^{b-1}(1+\alpha q^{x_{1}+j-i-k-\frac32})}.
\end{multline*}
That means we only need to prove that 
\begin{multline}
\label{eq:separate5}
\sum_{1 \leq a < b \leq k}q^{k-a+b}\prod_{j=1}^{k-2}\frac{1-q^j}{1-q}
 \frac{q^{(a-1)(k+i+\frac12)}}{\prod_{j=1}^{a-1}\frac{1-q^j}{1-q}\prod_{j=1}^{a-1}(1+\alpha q^{x_k+j-i-k+\frac{1}{2}})}\\ 
\times \sum_{\ell=0}^{b-a-1} \frac{(-1)^{\ell} q^{\sum_{j=a+1}^{a+\ell}(k-j+i+\frac32)}q^{(b-a-1-\ell)(k-a-\ell+i+\frac12)}}{\prod_{j=1}^{\ell}\frac{1-q^j}{1-q}\prod_{j=1}^{b-a-1-\ell}\frac{1-q^j}{1-q}\prod_{j=a+1}^{a+\ell}(1+\alpha q^{x_k+j-i-k-\frac12})\prod_{j=a+\ell+1}^{b-1}(1+\alpha q^{x_{1}+j-i-k-\frac32})}\\
\times \frac{\alpha^{k-b} q^{(k-b)x_1+\frac{(k-b)(k-b+1)}{2}}}{\prod_{j=1}^{k-b}\frac{1-q^j}{1-q}\prod_{j=b+1}^{k}(1+\alpha q^{x_1+j-i-k-\frac{5}{2}})}=q^{\frac{k(2k+2i-1)}{2}-2i}.
\end{multline}
When $a=1$ and $2 \leq b \leq k$, the terms in \eqref{eq:separate5} which contain only the variable $x_1$ have the following form
\begin{multline*}
\sum_{b=2}^{k}q^{k+b-1}\prod_{j=1}^{k-2}\frac{1-q^j}{1-q}
\frac{q^{(b-2)(k+i-\frac{1}{2})}}{\prod_{j=1}^{b-2}\frac{1-q^j}{1-q}\prod_{j=2}^{b-1}(1+\alpha q^{x_1+j-i-k-\frac{3}{2}})}\\
\times \frac{\alpha^{k-b} q^{(k-b)x_1+\frac{(k-b)(k-b+1)}{2}}}{\prod_{j=1}^{k-b}\frac{1-q^j}{1-q}\prod_{j=b+1}^{k}(1+\alpha q^{x_1+j-i-k-\frac{5}{2}})}=q^{\frac{k(2k+2i-1)}{2}-2i}.
\end{multline*}
where we use the following factorization formula: 
\begin{align*}
\prod_{j=1}^{k-2}(1+\alpha q^{x_1+j-i-k-\frac{1}{2}})=\sum_{b=2}^{k} \alpha^{k-b} q^{(k-b)x_1+\frac{b(b+2i)}{2}-\frac{k(k+2i)}{2}} \frac{\prod_{j=k-b+1}^{k-2}(1-q^j)}{\prod_{j=1}^{b-2}(1-q^j)}.
\end{align*}
When $1 \leq a < b \leq k$ and $1 \leq m \leq k-2$, the terms in \eqref{eq:separate5} which contain $m$ variables $x_k$  and $k-2-m$ variables $x_1$ in the denominator have the following form
\begin{multline*}
\sum_{1 \leq a \leq m+1}\ \sum_{m+2 \leq b \leq k}q^{k-a+b}\prod_{j=1}^{k-2}\frac{1-q^j}{1-q}
 \frac{q^{(a-1)(k+i+\frac12)}}{\prod_{j=1}^{a-1}\frac{1-q^j}{1-q}\prod_{j=1}^{a-1}(1+\alpha q^{x_k+j-i-k+\frac{1}{2}})}\\ 
\times \frac{(-1)^{m-a+1} q^{\sum_{j=a+1}^{m+1}(k-j+i+\frac{3}{2})}q^{(b-2-m)(k-m+i-\frac{1}{2})}}{\prod_{j=1}^{m+1-a}\frac{1-q^j}{1-q}\prod_{j=1}^{b-2-m}\frac{1-q^j}{1-q}\prod_{j=a+1}^{m+1}(1+\alpha q^{x_k+j-i-k-\frac{1}{2}})\prod_{j=m+2}^{b-1}(1+\alpha q^{x_1+j-i-k-\frac{3}{2}})}\\
\times\frac{\alpha^{k-b} q^{(k-b)x_1+\frac{(k-b)(k-b+1)}{2}}}{\prod_{j=1}^{k-b}\frac{(1-q^j)}{1-q}\prod_{j=b+1}^{k}(1+\alpha q^{x_1+j-i-k-\frac{5}{2}})}=0,
\end{multline*}
where we used the fact that for any fixed $m$ and $b$, one has
\begin{multline}
\label{eq:separate6}
\sum_{a=1}^{m+1}\frac{(-1)^{m-a+1}q^{-a}q^{(a-1)(k+i+\frac12)}q^{\sum_{j=a+1}^{m+1}(k-j+i+\frac{3}{2})}}{\prod_{j=1}^{a-1}(1-q^j)\prod_{j=1}^{m+1-a}(1-q^j)}\\
=\frac{(-1)^m q^{mk-\frac{m(m-2i)}{2}}}{\prod_{j=1}^{m}(1-q^j)}\sum_{a=1}^{m+1}(-1)^{a-1}q^{\frac{a(a-3)}{2}}\frac{\prod_{j=1}^{m}(1-q^j)}{\prod_{j=1}^{a-1}(1-q^j)\prod_{j=1}^{m+1-a}(1-q^j)}\\
=\frac{(-1)^m q^{mk-\frac{m(m-2i)}{2}}}{\prod_{j=1}^{m}(1-q^j)} \sum_{n=0}^{m}(-1)^{n}q^{\frac{(n+1)(n-2)}{2}}\frac{\prod_{j=1}^{m}(1-q^j)}{\prod_{j=1}^{n}(1-q^j)\prod_{j=1}^{m-n}(1-q^j)}\\
=\frac{(-1)^m q^{mk-\frac{m(m-2i)}{2}-1}}{\prod_{j=1}^{m}(1-q^j)}\prod_{n=0}^{m-1}(1-q^n)=0,
\end{multline}
indeed we use the  finite q-Binomial theorem \eqref{eq:q-Binomial}. This completes the proof. 
\end{proof}

\begin{remark}
\label{rmk:constrait}

For $x_{1}<x_{2}$, the expression for $\mathbb{P} \left(  \omega(i_1) =x_{1}, \omega(i_2)=x_{2} \right)$, $i_1 < i_2$, has a more complicated structure. In this remark, we provide the simplest example.
 
For the notation convenience, let us fix integers $x_{1}<x_{3}$. We consider $\mathbb{P} \left(  \omega(0) =x_{1}, \omega(2)=x_{3} \right)$, where $\omega$ is $\mathcal{M}^p_{\alpha}$-distributed. By definition, we have 
\begin{equation*}
\mathbb{P} \left(  \omega(0) =x_{1}, \omega(2)=x_{3} \right)
=\sum_{x_2 \in \mathbb{Z}} \mathbb{P} \left(  \omega(0) =x_{1}, \omega(1) =x_{2}, \omega(2)=x_{3}   \right)
\end{equation*}
Further, one has
\begin{multline*}
\mathbb{P} \left(  \omega(0) =x_{1}, \omega(2)=x_{3} \right)
=\sum_{x_2=x_3+1}^{\infty} \mathbb{P} \left(  \omega(0) =x_{1}, \omega(1) =x_{2}, \omega(2)=x_{3}   \right) \\
+\sum_{x_2=x_1+1}^{x_3-1} \mathbb{P} \left(  \omega(0) =x_{1}, \omega(1) =x_{2}, \omega(2)=x_{3} \right)
+\sum_{x_2=-\infty}^{x_1-1} \mathbb{P} \left(  \omega(0) =x_{1}, \omega(1) =x_{2}, \omega(2)=x_{3} \right)
\end{multline*}
Applying the result of Remark \ref{rmk:Pk} and simplifying, we obtain
\begin{multline*}
\mathbb{P} \left(  \omega(0) =x_{1}, \omega(2)=x_{3} \right)
=\frac{(1-q)^2\alpha^2 q^{x_1+x_3-6}}{(1+\alpha q^{x_1-\frac52})(1+\alpha q^{x_1-\frac32})(1+\alpha q^{x_3-\frac32})(1+\alpha q^{x_3-\frac12})}\\
+\frac{(1-q)^2\alpha^2 q^{x_1+x_3-6}(q^2-1)}{(1+\alpha q^{x_1-\frac52})(1+\alpha q^{x_1-\frac32})(1+\alpha q^{x_1-\frac12})(1+\alpha q^{x_3-\frac12})(1+\alpha q^{x_3+\frac12})}
\end{multline*}
It is easy to see that $\mathbb{P} \left(  \omega(0) =x_{1}, \omega(2)=x_{3} \right)$ has no simple product form in this case.
\end{remark}

\begin{remark}
\label{rmk:general}
Let $x_{1}, x_{2}, \cdots, x_{k}$ be $k$ integers and let $x_{1}>x_{2}>\cdots>x_{k}$. For the random permutation $\omega$ of $\mathbb{Z}$ distributed according to the Mallows product measure $\mathcal{M}^{p}_{\alpha}$ and any $k$ integers $i_1<i_2<\cdots<i_k$, we further consider the joint distribution of any $k$ elements of permutation $\omega(i_1), \omega(i_2), \cdots, \omega(i_k)$: 
\begin{align}
\label{eq:arbitrary}
\mathbb{P} \left(  \omega(i_1) =x_{1}, \omega(i_2)=x_{2}, \cdots,  \omega(i_k)=x_{k} \right), 
\end{align}
Similar with Proposition \ref{prop:separate}, we can also regard \eqref{eq:arbitrary} as the marginal distribution of $\mathbb{P} \left(  \omega(i_1), \omega(i_1+1), \cdots,  \omega(i_k-1), \omega(i_k)\right)$. All the steps in the proof of Proposition \ref{prop:separate} work except of the analog of \eqref{eq:inv-summation}:
\begin{align}
\label{eq:multi-inv-summation}
\sum_{\{\sigma \in S_k:\ \sigma(a_1)=i_k, \sigma(a_2)=i_{k-1}, \cdots,  \sigma(a_k)=i_1 \}}q^{inv(\sigma)}, 
\ \ 1 \leq a_1 < a_2 < \cdots <a_k \leq k.
\end{align}
We were unable to write down an explicit formula for the summation \eqref{eq:multi-inv-summation} or treat these expressions in a sufficient for our purposes way. As we prove in the next section, the probability \eqref{eq:arbitrary} does have a simple product form --- however, we prove it via a different route. The expressions like \eqref{eq:multi-inv-summation} also provide the key difficulty in attempts to write down an explicit formula for the probability \eqref{eq:arbitrary} for arbitrary linear orderings of $x_i$'s. 

\end{remark}

\section{Hidden symmetries of Mallows product measure}
\label{sec:Symmetries}

In this section we establish certain distributional identities for various observables of the Mallows (product) measure. To a large extent, this is an application of general results from the works of Borodin-Gorin-Wheeler \cite{BGW22} and Galashin \cite{Gal21} on the shift-invariance / hidden symmetries of the stochastic colored six-vertex model.  Combining the identities with the results of the previous Section, we prove Theorems \ref{thm:arbitrary} and \ref{thm:correlation}.

\subsection{Hidden symmetries of infinity-species ASEP}

In this section, we briefly recall the general result \cite[Theorem 1.6]{Gal21}, which is a distributional identity of certain observables of the stochastic colored six-vertex model, and apply it to the infinite-species ASEP via the known degeneration of the former to the latter.  

Similar applications of the symmetry of the stochastic colored six-vertex model to the symmetry of infinity-species TASEP/ASEP have appeared in \cite{BGR22, Zha22, Zha23}, and in \cite{H22} for the half-space case. The limit transition from the stochastic six-vertex model to the ASEP was observed in \cite{BCG16}, and proved in full detail in \cite{Agg17} (see also \cite{BBCW18, Yan22}).

Let us recall the notation for the stochastic colored six-vertex model used in \cite{Gal21}. We consider the \emph{stochastic colored six-vertex model} as the random ensemble of colored up-right paths in the positive quadrant $(x,y) \in \mathbb{Z}_{\geq 0} \times \mathbb{Z}_{\geq 1}$ with the colors labeled by integers $\mathbb{Z}$. We assume that no lattice edge can be occupied by more than one path. We begin with a path of color $x$ entering the vertex $(x, 1)$ from the bottom, and a path of color $-y$ entering the vertex $(0, y)$ from the left. Given the entering paths, they progress in the up-right direction within the quadrant. For each vertex of the lattice, given the colors of the entering paths along the bottom and left adjacent edges, we choose the colors of the exiting paths along the top and right edges according to the following probabilities (see Figure \ref{fig:6v}). The model depends on a quantization parameter $q \in (0, 1)$ and real column and row \emph{rapidities} denoted by $\bm{u}=(u_0, u_1, \ldots)$ and $\bm{v}=(v_1, v_2, \ldots)$, respectively. The rapidities are assigned to the corresponding rows and columns, and we assume that $v_y \geq u_x \geq 0$ for all $x \geq 0$ and $y \geq 1$.
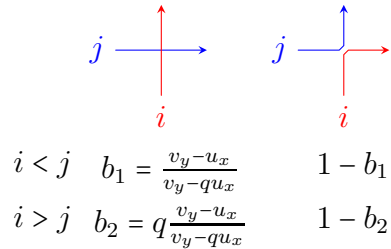
\begin{figure}[hbt!]
    \centering
\begin{tikzpicture}[line cap=round,line join=round,>=triangle 45,x=1.2cm,y=1.2cm]
\clip(-1.2,4.7) rectangle (3.1,7.6);
\draw [-stealth,color=blue](0,7) -- (1,7);
\draw [-stealth,color=red](0.5,6.5) -- (0.5,7.5);
\draw [-stealth,color=blue](2,7) -- (2.45,7) -- (2.5,7.05) -- (2.5,7.5);
\draw [-stealth,color=red](2.5,6.5) -- (2.5,6.95) -- (2.55,7) -- (3,7);
\draw (0,7) node[anchor=east, color=blue]{$j$};
\draw (0.5,6.5) node[anchor=north, color=red]{$i$};
\draw (2,7) node[anchor=east, color=blue]{$j$};
\draw (2.5,6.5) node[anchor=north, color=red]{$i$};
\draw (-0.8,6) node[anchor=north]{$i< j$};
\draw (-0.8,5.4) node[anchor=north]{$i> j$};
\draw (0.6,6) node[anchor=north]{$b_1=\frac{v_y-u_x}{v_y-qu_x}$};
\draw (2.6,6) node[anchor=north]{$1-b_1$};
\draw (0.6,5.4) node[anchor=north]{$b_2=q\frac{v_y-u_x}{v_y-qu_x}$};
\draw (2.6,5.4) node[anchor=north]{$1-b_2$};
\end{tikzpicture}
\caption{Probabilities at each vertex of the stochastic colored six-vertex model. }  
\label{fig:6v}
\end{figure}
We consider the stochastic colored six-vertex model in a finite domain. A \emph{skew domain} is the set bounded by a pair $(P, Q)$ of up-left paths with common start and end points such that P is weakly below and to the left of Q ( here and below we refer to \cite[Section 1]{Gal21} for more detail). A \emph{$(P,Q)$-cut} is a quadruple $C=(l,b,r,t)$ of positive integers satisfying  $l\leq r$, $b\leq t$, and such that $(l,b),(r,t)\in (P, Q)$ while $(l-1,b-1),(r+1,t+1) \notin (P, Q)$. Note that $C$ defines a rectangle with corners $(l, b)$ and $(r, t)$, and these points must lie on P and Q respectively. We now define the notion of \emph{height functions} in a skew domain $(P, Q)$. We will let $Ht^{P, Q}(C)$ denote the number of colored paths that connect the left and right boundaries of the sub-rectangle $\{l,l+1,\dots,r\}\times\{b,b+1,\dots,t\}$ of $(P, Q)$(this is equivalent to the number of colors entering $C$ from the left and exiting $C$ from the right). If the dependence on $\bm{u}$ and $\bm{v}$ is relevant, we will write $Ht^{P, Q}(C; \bm{u},\bm{v})$. Let $supp_{H}(C):=\{u_l,u_{l+1},\dots,u_{r}\}$ and $supp_{V}(C):=\{v_b,v_{b+1},\dots,v_{t}\}$ denote the \emph{unordered} sets of column and row rapidities covered by $C$. Let us say that $\bm{u}'=(u'_0,u'_1,\dots)$ is a \emph{permutation of the variables in $\bm{u}$} if there exists a bijection $\phi:\mathbb{Z}_{\geq 0}\to\mathbb{Z}_{\geq 0}$ such that $u'_i=u_{\phi(i)}$ for all $i\in\mathbb{Z}_{\geq 0}$. We can also define $\bm{v}'$ as a \emph{permutation of the variables in $\bm{v}$} in a similar way. The result \cite[Theorem 1.6]{Gal21} is as follows, in the notations above. 

\begin{theorem}\label{thm:main}
Suppose that we are given the following data:
\begin{itemize}
\item two skew domains $(P,Q)$ and $(P',Q')$;
\item a permutation $\bm{u}'$ of the variables in $\bm{u}$ and a permutation $\bm{v}'$ of the variables in $\bm{v}$;
\item a tuple $(C_1,C_2,\dots,C_m)$ of $(P,Q)$-cuts and a tuple $(C'_1,C'_2,\dots,C'_m)$ of $(P',Q')$-cuts.
\end{itemize}
Assume that for each $i=1,2,\dots,m$, we have 
\begin{equation}
\label{eq:main:supp=supp}
 supp_{H}(C_i;\bm{u})=supp_{H}(C'_i;\bm{u}')\quad\text{and}\quad supp_{V}(C_i;\bm{v})=supp_{V}(C'_i;\bm{v}').
\end{equation} 
Then the  distributions of the following two vectors of height functions agree:
\begin{align}
\label{eq:shift-invariance-SCSVM}
\Big(Ht^{P,Q}(C_1;\bm{u},\bm{v}),\dots,Ht^{P,Q}(C_m;\bm{u},\bm{v})\Big) \stackrel{d}{=} \Big(Ht^{P,Q}(C'_1;\bm{u}',\bm{v}'),\dots,Ht^{P,Q}(C'_m;\bm{u}',\bm{v}')\Big).
\end{align}
\end{theorem}

We next state the limit transition from the stochastic colored six-vertex model to the infinity-species ASEP. For this, we introduce the height function of the infinity-species ASEP. Let $\pi^{step}_t$ be the infinity-species ASEP started from the initial configuration $\pi^{step}_0(i)=i$, $i \in \mathbb{Z}$. For any $i,j \in \Z +\frac12$, we define the \emph{height function} $h_{<i \to >j} \left( t\right)$ via
\begin{align*}
h_{<i \to >j} \left( t \right) := \#\{\mbox{ $a \in \Z$: $a>j$ and $\pi^{step}_t(a) < i$}\}.
\end{align*}
In words, $h_{<i \to >j} \left( t \right)$ counts the number of particles of color $<i$ at site $>j$ at time $t$. In principle, $h_{<i \to >j} \left( t \right)$ can take infinite values, but for the random permutation $\pi^{step}_t$ that we consider it will be almost surely finite. 

The homogeneous (all $\bm{u}$-rapidities and, separately, all $\bm{v}$-rapidities are equal to each other, so each vertex has the same parameters $b_1,b_2$) stochastic colored six-vertex model on $\mathbb{Z}_{\geq 0} \times \mathbb{Z}_{\geq 1}$ can be thought of as a discrete-time interacting particle system by placing a particle colored $i$ at site $p$ and time $t-1$ if and only if a path colored $i$ vertically enters through the vertex $(p, t)$ (with parameters $b_1, b_2$). We denote the position of the particle colored $i\in\mathbb{Z}$ at time $t$ by $p_{i}(t)$.

\begin{theorem}
\label{thm:SCSVM-converge-CASEP}

In the discrete-time interacting particle system described above, consider the scaling of positions $q_i (T) = p_i (T) -T$, parameters $(b_1,b_2)=(\epsilon, q \epsilon)$ and time $T=\lfloor \epsilon^{-1} t \rfloor$. In the limit $\epsilon \to 0$, the system at time $T$ converges weakly to the continuous-time infinity-species ASEP at time $t$ with the initial configuration $\pi^{step}_0(i)=i$, $i \in \mathbb{Z}$. 
Further, let us choose a skew domain $(P,Q)$ from Theorem \ref{thm:main} as the rectangle with vertices $(0,0)$, $(2T,0)$, $(2T,T)$, $(0,T)$ and cuts $C_i$ from the statement of Theorem \ref{thm:main} as the lines connecting points $(\hat x_i - \frac12,0)$ and $(T+ \hat y_i +\frac12,T)$, where $\hat{x}_i, \hat{y}_i \in \mathbb{Z}+\frac12$, $i=1,2,\cdots, m$(see Figure \ref{fig:cuts}). Then we have the following result:
\begin{align}
\label{eq:height-converge}
\Big(Ht^{P,Q}(C_1),\dots,Ht^{P,Q}(C_m)\Big) \xrightarrow[\epsilon \to 0]{} \Big(h_{<\hat{x}_1 \to > \hat{y}_1} \left( t \right),\dots, h_{<\hat{x}_m \to > \hat{y}_m} \left( t \right)\Big)
\end{align}
\end{theorem}

\begin{proof}
The proof is analogous to \cite{Agg17, AB24}. Very similar statements were used e.g. in \cite{BGR22,H22}.
\end{proof}

\begin{figure}
\begin{center}
\begin{tikzpicture}[scale=1.4]
\draw[lgray,line width=2pt] (0,0) -- (0,5) -- (10,5) -- (10,0) -- (0,0);
\draw[black,line width=1pt, dashed] (0,0) -- (5,5);
\draw[blue,line width=1pt] (1,0) -- (7.5,5);
\draw[blue,line width=1pt] (2,0) -- (5.5,5);
\draw[blue,line width=1pt] (2.5,0) -- (4.5,5);
\draw (-0.2,-0.2) node {$(0,0)$};
\draw (-0.2,5.2) node {$(0,T)$};
\draw (10.2,-0.2) node {$(2T,0)$};
\draw (10.2,5.2) node {$(2T,T)$};
\draw (1,-0.2) node[scale=0.6,blue] {$\hat x_m - \frac12$};
\draw (1.5,-0.2) node[scale=0.6,blue] {$\cdots$};
\draw (2,-0.2) node[scale=0.6,blue] {$\hat x_2 - \frac12$};
\draw (2.5,-0.2) node[scale=0.6,blue] {$\hat x_1 - \frac12$};
\draw (4.5,5.2) node[scale=0.6,blue] {$T+ \hat y_1 +\frac12$};
\draw (5.5,5.2) node[scale=0.6,blue] {$T+ \hat y_2 +\frac12$};
\draw (6.5,5.2) node[scale=0.6,blue] {$\cdots$};
\draw (7.5,5.2) node[scale=0.6,blue] {$T+ \hat y_m +\frac12$};
\end{tikzpicture}
\end{center}
\caption{Points involved in the limit transition from Theorem \ref{thm:SCSVM-converge-CASEP}: The rectangle with vertices $(0,0)$, $(2T,0)$, $(2T,T)$, $(0,T)$ and cuts $C_i$ from points $(\hat x_i - \frac12,0)$ to $(T+ \hat y_i +\frac12,T)$, where $\hat{x}_i, \hat{y}_i \in \mathbb{Z}+\frac12$, $i=1,2,\cdots, m$.}
\label{fig:cuts}
\end{figure}
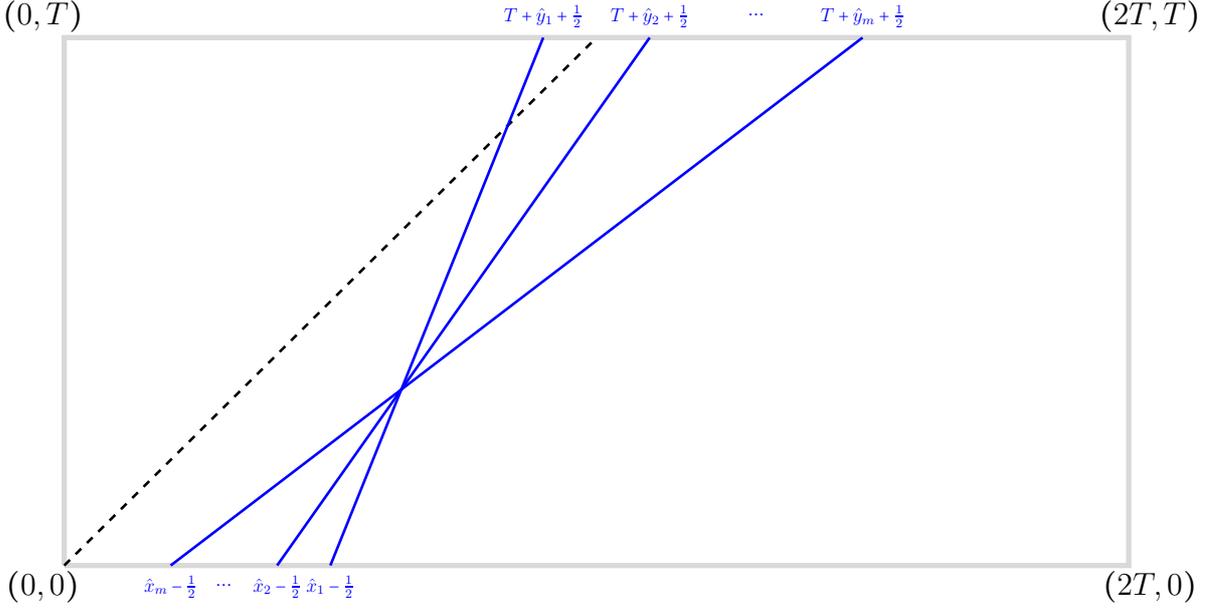

Our next goal is to formulate the hidden symmetries of the infinity-species ASEP coming from Theorem \ref{thm:main}. For this, we need more notations. Let $A,B,C,D \in \Z +\frac12$ and $A \le B$, $C \le D$. Let $\{ \hat x_i \}_{i=1}^N$, $\{ \tilde x_i \}_{i=1}^N$, $\{ \hat y_i \}_{i=1}^N$, and $\{ \tilde y_i \}_{i=1}^N$ be sequences of numbers from $\Z +\frac12$ such that $A \leq \hat x_i \leq B$, $A \leq \tilde x_i \leq B$, $C \leq \hat y_i \leq D$, $C \leq \tilde y_i \leq D$ for $i=1, \dots, N$, where $N \in \N$. It would be convenient for us to fix an integer $S$ such that $B < S+C$; the statement below does not depend on a specific choice of $S$. We denote by $\mathcal A$ the set 
\[
\Z \cap \left( (A;B) \cup (S+C;S+D) \right),
\]
where our condition on $S$ guarantees that this is the union of two disjoint intervals of integers. 

For a height function $h_{< \hat x_i \to > \hat y_i} \left( \pi \right)$ we associate its \textit{support} $\mathrm{supp}( \hat x_i, \hat y_i)$ defined by 
\[
\mathrm{supp}( \hat x_i, \hat y_i) :=  \{ a \in \mathcal{A} :  \hat x_i < a <  S+ \hat y_i  \}.
\]
Let $\mathfrak{g}$ be a permutation of $\mathcal A$, and define
\[
\mathrm{supp}_{\mathfrak{g}} (\tilde x_i, \tilde y_i) := \{ b \in \mathcal{A} : \tilde x_i < \mathfrak{g} (b) < S+\tilde y_i  \}.
\]

\begin{example}
\label{ex:shift-invar1}
Let $(\hat x_1, \hat x_2)= ( \frac12, \frac52)$, $(\hat y_1, \hat y_2)= ( \frac92, \frac72)$, $(\tilde x_1, \tilde x_2)= ( \frac12, \frac32)$, $(\tilde y_1, \tilde y_2)= ( \frac92, \frac52)$. One can choose $A=\frac12$, $B=\frac52$, $C=\frac52$, $D=\frac92$, $S=2$ (any other $S \in \N$ would also work in the same way). With these choices, we have $\mathcal A = \{1,2,5,6\}$, and 
\[
\mathrm{supp}( \hat x_1, \hat y_1) = \{1,2,5,6\},
\quad \mathrm{supp}( \hat x_2, \hat y_2) = \{ 5 \}.
\]
Finally, let $\mathfrak{g}$ be a permutation of $\mathcal A$ such that $\mathfrak{g}(1)=1$, $\mathfrak{g}(2)=5$, $\mathfrak{g}(5)=2$, $\mathfrak{g}(6)=6$. Then 
\[
\mathrm{supp}_{\mathfrak{g}} ( \tilde x_1, \tilde y_1) = \{1,2,5,6\}, \quad \mathrm{supp}_{\mathfrak{g}}( \tilde x_2, \tilde y_2) = \{ 5 \}.
\]

\end{example}

\begin{theorem}
\label{th:shift-invariance}

In notations above, assume that for any $1 \le i \le N$ one has 
\begin{align}
\label{eq:supp}
\mathrm{supp} ( \hat x_i, \hat y_i) = \mathrm{supp}_{\mathfrak{g}} (\tilde x_i, \tilde y_i).
\end{align}
Then one has the following equality in distribution of random vectors
\begin{multline}
\label{eq:shift-invariance-CASEP}
\left( h_{< \hat x_1 \to > \hat y_1} \left( \pi^{step}_t \right), h_{< \hat x_2 \to > \hat y_2} \left( \pi^{step}_t \right), \dots, h_{< \hat x_N \to > \hat y_N} \left( \pi^{step}_t \right) \right) \\ 
\stackrel{d}{=} \left( h_{< \tilde x_1 \to > \tilde y_1} \left( \pi^{step}_t \right), h_{< \tilde x_2 \to > \tilde y_2} \left( \pi^{step}_t \right), \dots, h_{< \tilde x_N \to >\tilde y_N} \left( \pi^{step}_t \right) \right).
\end{multline}
\end{theorem}

\begin{proof}
Let $T \in \N$ be a large parameter. We choose both skew domains from Theorem \ref{thm:main} as the rectangle with vertices $(0,0)$, $(2T,0)$, $(2T,T)$, $(0,T)$. For convenience, let us assume first that $A>0$. Then we choose the cuts $C_i$ in the statement of Theorem \ref{thm:main} as the lines connecting points $( \hat x_i - \frac12,0)$ and $(T+ \hat y_i + \frac12,T)$, and choose the cuts $C'_i$ as the lines connecting points $(\tilde x_i - \frac12,0)$ and $(T + \tilde y_i +  \frac12,T)$. When $T$ satisfies $B < T+C$, the equality of supports of such cuts required in Theorem \ref{thm:main} corresponds to the equality of supports \eqref{eq:supp} in the sense defined above. Therefore, Theorem \ref{thm:main} provides the equality in distribution for two such collections of cuts. 

Take $b_1=\epsilon$, $b_2=q\epsilon$, and scale $T$ by $\lfloor \epsilon^{-1} t \rfloor$ in the stochastic colored six vertex model. When $\epsilon \to 0$, we can get equation \eqref{eq:shift-invariance-CASEP} from equation \eqref{eq:shift-invariance-SCSVM} by Theorem \ref{thm:SCSVM-converge-CASEP}. Finally, the condition $A>0$ can be removed due to the translation-invariance of the continuous time infinite-species ASEP, which concludes the proof. 

\end{proof}

\begin{example}
As a continuation of Example \ref{ex:shift-invar1}, we see that the data from it satisfies the assumptions of the theorem. Therefore, we can conclude that the infinite-species ASEP started from the step initial condition satisfies the following distribution identity:
\begin{equation*}
\left( h_{< \frac12 \to > \frac92} \left( \pi^{step}_t \right), h_{< \frac52 \to > \frac72} \left( \pi^{step}_t \right) \right) \\ 
\stackrel{d}{=} \left( h_{< \frac12 \to > \frac92} \left( \pi^{step}_t \right), h_{< \frac32 \to > \frac52} \left( \pi^{step}_t \right) \right).
\end{equation*}

\end{example}

\begin{remark}
A slightly artificial formulation of Theorem \ref{th:shift-invariance} (the presence of $S$) seems to be due to the fact that these hidden symmetries look a bit more natural in the discrete setting of the stochastic six vertex model \cite{Gal21} or directed last passage percolation \cite{Dau22} in comparison with a continuous time limit. An alternative way to formulate the conditions needed for the symmetries uses \textit{intersection matrices} from \cite[Section 6.3]{Gal21}. Nevertheless, this theorem provides a powerful method for the study of the infinite-species ASEP. In the current paper, we apply it to the study of a reversible stationary distribution of the infinite-species ASEP.
\end{remark}

\subsection{Shift-invariance property of infinite-species ASEP}
\label{ssec:Shift-invariance-ASEP}

\begin{proposition}
\label{prop:shift-invariance-ASEP}
Let $\pi^{step}_t$ be a configuration of a infinite-species ASEP after time $t$ started from $\pi^{step}_0 (i)=i$, $i \in \Z$. Let $k \in \N$ and let $i_1 < i_2 < \cdots < i_k$ , $x_{1}>x_2> \cdots >x_k$ be integers. For any $t \ge 0$ one has
\begin{multline}
\mathbb{P} \left(  \pi^{step}_t (i_1) =x_1, \pi^{step}_t (i_2)=x_2, \cdots,  \pi^{step}_t (i_k)=x_k \right) \\ =
\mathbb{P} \left(  \pi^{step}_t (i_1) =x_1, \pi^{step}_t (i_1+1)=x_2-i_2+i_1+1, \cdots,  \pi^{step}_t (i_1+k-1)=x_k-i_k+i_1+k-1 \right).
\end{multline}
See Figure \ref{fig:shift-invariance} for illustration.
\end{proposition}

\begin{proof}

Note that the event $\pi^{step}_t (i)=x$ is equivalent to the event
\begin{equation}
\label{eq:height-corr}
h_{< i- \frac12 \to > x+\frac12} ( \pi^{step}_t ) - h_{< i + \frac12 \to > x+\frac12} ( \pi^{step}_t ) - h_{< i- \frac12 \to > x - \frac12} ( \pi^{step}_t ) + h_{< i + \frac12 \to > x - \frac12} ( \pi^{step}_t ) =1.
\end{equation}

We apply Theorem \ref{th:shift-invariance} for the following data: $N= 4k$, $A=i_1 - \frac12$, $B=i_k + \frac12$ , $C=x_k - \frac12$, $D=x_1 + \frac12$, and for any $d=1,2, \dots, k$ set $\hat x_{4d-3} = i_d - \frac12$, $\hat x_{4d-2} = i_d - \frac12$, $\hat x_{4d-1} = i_d + \frac12$, $\hat x_{4d} = i_d + \frac12$, $\hat y_{4d-3} = x_d - \frac12$, $\hat y_{4d-2} = x_d + \frac12$, $\hat y_{4d-1} = x_d - \frac12$, $\hat y_{4d} = x_d + \frac12$, 
 $\tilde x_{4d-3} = i_1+d-1 - \frac12$, $\tilde x_{4d-2} = i_1+d-1 - \frac12$, $\tilde x_{4d-1} = i_1+d-1 + \frac12$, $\tilde x_{4d} = i_1+d-1 + \frac12$, $\tilde y_{4d-3} = x_d-i_d+i_1+d-1 - \frac12$, $\tilde y_{4d-2} = x_d-i_d+i_1+d-1 + \frac12$, $\tilde y_{4d-1} = x_d-i_d+i_1+d-1 - \frac12$, $\tilde y_{4d} = x_d-i_d+i_1+d-1 + \frac12$ . We choose arbitrary $S \in \N$ such that $S+C >B$, and let $\mathcal A$ be as defined in the assumptions of Theorem \ref{th:shift-invariance}.

We need to point a permutation $\mathfrak{g}$ of $\mathcal A$ that we use. For this, we split $\mathcal A$ into intervals, which we list in their linear order: $\{ i_1 \}$, $(i_1;i_2)$, $\{ i_2 \}$, $(i_2;i_3)$, ..., $(i_{k-1};i_k)$, $\{ i_k \}$, $\{ S+x_k \}$, $(S+x_k;S+x_{k-1})$, $\{ S+x_{k-1} \}$, $(S+x_{k-1};S+x_{k-2})$, ..., $(S+x_2;S+x_1)$, $\{ S+x_1 \}$, where we denote by e.g. $(i_1;i_2)$ only integer points from this open interval. The permutation $\mathfrak{g}$ keeps the integers inside these intervals in the same order, and permute the intervals such that the images of the intervals are linearly ordered in the following way: $\mathfrak{g} \left( \{ i_1 \} \right)$, $\mathfrak{g} \left( \{ i_2 \} \right)$, ..., $\mathfrak{g} \left( \{ i_k \} \right)$, $\mathfrak{g} \left( \{ S+x_k \} \right)$, $\mathfrak{g} \left( \{ (i_{k-1};i_k)  \} \right)$, $\mathfrak{g} \left( \{ (S+x_k;S+x_{k-1})  \} \right)$, $\mathfrak{g} \left( \{ S+x_{k-1} \} \right)$, $\mathfrak{g} \left( \{ (i_{k-2};i_{k-1})  \} \right)$, $\mathfrak{g} \left( \{ (S+x_{k-1};S+x_{k-2})  \} \right)$, ... , $\mathfrak{g} \left( \{ S+x_{2} \} \right)$, $\mathfrak{g} \left( \{ (i_{1};i_{2})  \} \right)$, $\mathfrak{g} \left( \{ (S+x_{2};S+x_{1})  \} \right)$,  $\mathfrak{g} \left( \{ S+x_{1} \} \right)$. These conditions determine $\mathfrak{g}$ in a unique way. For example, one has $\mathfrak{g} (i_2)=i_1+1$, $\mathfrak{g} (i_3)=i_1+2$, $\mathfrak{g} (S+x_1)=S+x_1$.

From the construction of $\mathfrak{g}$, for any $d=1,2,\dots,k$, one has 
\[
\mathrm{supp}_{\mathfrak{g}} (\tilde x_{4d-3}, \tilde y_{4d-3}) = \{ i_d, i_{d+1}, \dots, i_k, S+x_k, S+x_k-1, \dots, S+x_d-1 \} = \mathrm{supp} ( \hat x_{4d-3}, \hat y_{4d-3}),
\]
with analogous equalities for indices of the form $4d-2$, $4d-1$, and $4d$. Therefore, one can apply Theorem \ref{th:shift-invariance} to this data and get the distributional identity of corresponding vectors of height functions. Combining it with \eqref{eq:height-corr}, we arrive at the statement of the proposition.  
\end{proof}

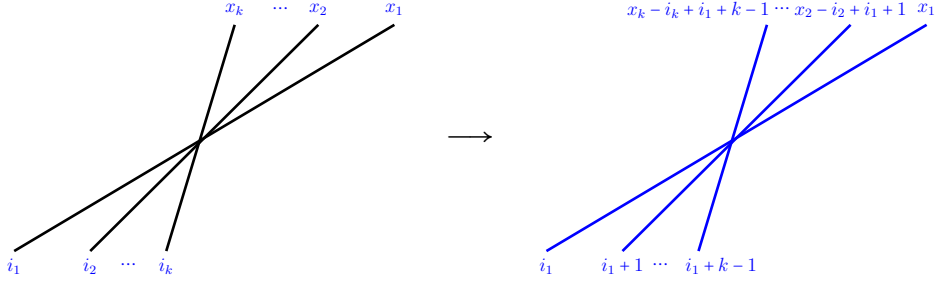
\begin{figure}
\begin{center}
\begin{tikzpicture}[scale=1]
\draw[black,line width=1pt] (0,0) -- (5,3);
\draw[black,line width=1pt] (1,0) -- (4,3);
\draw[black,line width=1pt] (2,0) -- (2.9,3);
\draw (6,1.5) node  {$\longrightarrow$};
\draw[blue,line width=1pt] (7,0) -- (12,3);
\draw[blue,line width=1pt] (8,0) -- (11,3);
\draw[blue,line width=1pt] (9,0) -- (9.9,3);
\draw (0,-0.2) node[scale=0.6,blue] {$i_1$};
\draw (1,-0.2) node[scale=0.6,blue] {$i_2$};
\draw (1.5,-0.2) node[scale=0.6,blue] {$\cdots$};
\draw (2,-0.2) node[scale=0.6,blue] {$i_k$};
\draw (2.9,3.2) node[scale=0.6,blue] {$x_k$};
\draw (3.5,3.2) node[scale=0.6,blue] {$\cdots$};
\draw (4,3.2) node[scale=0.6,blue] {$x_2$};
\draw (5,3.2) node[scale=0.6,blue] {$x_1$};
\draw (7,-0.2) node[scale=0.6,blue] {$i_1$};
\draw (8,-0.2) node[scale=0.6,blue] {$i_1+1$};
\draw (8.5,-0.2) node[scale=0.6,blue] {$\cdots$};
\draw (9.3,-0.2) node[scale=0.6,blue] {$i_1+k-1$};
\draw (9,3.2) node[scale=0.6,blue] {$x_k-i_k+i_1+k-1$};
\draw (10.1,3.2) node[scale=0.6,blue] {$\cdots$};
\draw (11,3.2) node[scale=0.6,blue] {$x_2-i_2+i_1+1$};
\draw (12,3.2) node[scale=0.6,blue] {$x_1$};
\end{tikzpicture}
\end{center}
\caption{Points involved in the shift invariance of the infinite-species ASEP, see Proposition \ref{prop:shift-invariance-ASEP}.}
\label{fig:shift-invariance}
\end{figure}

\begin{theorem}
\label{thm:MPM-shift}
As before, let $\omega$ be the random permutation $\Z \to \Z$ distributed according to the Mallows product measure. Let $k \in \N$ and let $i_1 < i_2 < \cdots < i_k$ , $x_{1}>x_2> \cdots >x_k$ be integers. For any $t \ge 0$ one has
\begin{multline}
\label{eq:MPM-shift}
\mathbb{P} \left(  \omega (i_1) =x_1, \omega (i_2)=x_2, \cdots,  \omega (i_k)=x_k \right) \\ =
\mathbb{P} \left(  \omega (i_1) =x_1, \omega (i_1+1)=x_2-i_2+i_1+1, \cdots, \omega (i_1+k-1)=x_k-i_k+i_1+k-1 \right).
\end{multline}
\end{theorem}

\begin{proof}
Let us consider $t \to \infty$ limit of the claim of Proposition \ref{prop:shift-invariance-ASEP}. Due to Proposition \ref{prop:converg-to-stat}, we obtain the statement of the current Proposition for the Mallows measure $\mathcal{M}_0$. Therefore, the statement holds also for measures $\{ \mathcal{M}_c \}$ and their linear combination --- the Mallows product measure. 

\end{proof}

Now we are in a position to prove the shift-invariance property of the Mallows product measure:
\begin{theorem}
\label{thm:M-shift-invariance}
Let $i_1 < i_2 < \cdots < i_k$, $x_{1}>x_2> \cdots >x_k$, $j_1 < j_2 < \cdots < j_k$, $y_{1}>y_2> \cdots >y_k$ be integers such that $x_a - i_a = y_a - j_a$, for any $a=1,2,\dots, k$, and $k \in \N$. Let $w$ be the random permutation of $\Z$ distributed according to the Mallows product measure. We have
\begin{multline}
\label{eq:shift-invariance}
\mathbb{P} \left(  \omega(i_1) =x_1, \omega(i_2)=x_2, \cdots,  \omega(i_k)=x_k \right) \\ =
\mathbb{P} \left(  \omega(j_1) =y_1, \omega(j_2)=y_2, \cdots,  \omega(j_k)=y_k \right)
\end{multline}
\end{theorem}

\begin{proof}
By Theorem \ref{thm:MPM-shift} and the translation invariance, we have 
\begin{multline*}
\mathbb{P} \left(  \omega(i_1) =x_1, \omega(i_2)=x_2, \cdots,  \omega(i_k)=x_k \right) \\ 
=\mathbb{P} \left(  \omega(i_1) =x_1, \omega(i_1+1)=x_2-i_2+i_1+1, \cdots,  \omega(i_1+k-1)=x_k-i_k+i_1+k-1 \right)\\
=\mathbb{P} \left(  \omega(0) =x_1-i_1, \omega(1)=x_2-i_2+1, \cdots,  \omega(k-1)=x_k-i_k+k-1 \right)
\end{multline*}
and 
\begin{multline*}
\mathbb{P} \left(  \omega(j_1) =y_1, \omega(j_2)=y_2, \cdots,  \omega(j_k)=y_k \right) \\ 
=\mathbb{P} \left(  \omega(j_1) =y_1, \omega(j_1+1)=y_2-j_2+j_1+1, \cdots,  \omega(j_1+k-1)=y_k-j_k+j_1+k-1 \right)\\
=\mathbb{P} \left(  \omega(0) =y_1-j_1, \omega(1)=y_2-j_2+1, \cdots,  \omega(k-1)=y_k-j_k+k-1 \right)
\end{multline*}
then the condition $x_a - i_a = y_a - j_a$ for $a=1,2,\dots, k$ implies \eqref{eq:shift-invariance}.
\end{proof}

By the approach above, we can also analyze another observable of the Mallows product measure. We provide only a sketch of the proof of the following theorem, since all the steps are analogous to the proof of Theorem \ref{thm:M-shift-invariance}.

\begin{theorem}
\label{thm:G-shift-invariance}
Let $i_1 < i_2 < \cdots < i_k$, $x_{1} \ge x_2 \ge \cdots \ge x_k$, $j_1 < j_2 < \cdots < j_k$, $y_{1} \ge y_2 \ge \cdots \ge y_k$ be integers such that $x_a - i_a = y_a - j_a$, for any $a=1,2,\dots, k$, and $k \in \N$. Let $w$ be the random permutation of $\Z$ distributed according to the Mallows product measure. We have
\begin{multline}
\label{eq:G-shift-invariance}
\mathbb{P} \left(  \omega(i_1) \leq x_1, \omega(i_2) \leq x_2, \cdots,  \omega(i_k) \leq x_k \right) \\ =
\mathbb{P} \left(  \omega(j_1) \leq y_1, \omega(j_2) \leq y_2, \cdots,  \omega(j_k) \leq y_k \right).
\end{multline}
\end{theorem}

\begin{proof}
First, one needs to prove a similar to Proposition \ref{prop:shift-invariance-ASEP} result: 
\begin{multline}
\mathbb{P} \left(  \pi^{step}_t (i_1) \leq x_1, \pi^{step}_t (i_2) \leq x_2, \cdots,  \pi^{step}_t (i_k) \leq x_k \right) \\ =
\mathbb{P} \left(  \pi^{step}_t (i_1) \leq x_1, \pi^{step}_t (i_1+1) \leq x_2-i_2+i_1+1, \cdots,  \pi^{step}_t (i_1+k-1) \leq x_k-i_k+i_1+k-1 \right).
\end{multline}
Similar with the proof of Proposition \ref{prop:shift-invariance-ASEP}, we use Theorem \ref{th:shift-invariance} and the analog of \eqref{eq:height-corr}:
 \[ 
\mathbb{P} \left( \pi^{step}_t (i) \leq x \right) = \mathbb{P} \left( h_{<i+\frac12 \to >x+\frac12}-h_{<i-\frac12 \to >x+\frac12}=0 \right).
\]
To use Theorem \ref{th:shift-invariance} one produces the permutation and the rest of the data in a very similar way, we omit the precise description for brevity. Using Proposition \ref{prop:converg-to-stat}, we can get the following result:
\begin{multline}
\mathbb{P} \left(  \omega (i_1) \leq x_1, \omega (i_2) \leq x_2, \cdots,  \omega (i_k) \leq x_k \right) \\ =
\mathbb{P} \left(  \omega (i_1) \leq x_1, \omega (i_1+1) \leq x_2-i_2+i_1+1, \cdots, \omega (i_1+k-1) \leq x_k-i_k+i_1+k-1 \right).
\end{multline}
Analogously to the proof of Theorem \ref{thm:M-shift-invariance}, we get \eqref{eq:G-shift-invariance}.
\end{proof}

\subsection{Observables of Mallows product measure}
\label{ssec:arbitrary-separate}

Now we can prove a product formula for the (partial) joint distribution of arbitrary images of integers under the Mallows product measure.

\begin{theorem}
\label{prop:arbitrary}
Let $x_{1}, x_{2}, \cdots, x_{k}$ be $k$ integers, $x_{1}>x_{2}>\cdots>x_{k}$. For the random permutation $\omega$ of $\mathbb{Z}$ distributed according to the Mallows product measure $\mathcal{M}^{p}_{\alpha}$ and any $k$ integers $i_1<i_2<\cdots<i_k$, one has
\begin{equation}
\label{eq:separate}
\mathbb{P} \left(  \omega(i_1) =x_{1}, \omega(i_2)=x_{2}, \cdots,  \omega(i_k)=x_{k} \right) \\
=\prod_{j=1}^{k}\frac{(1-q) \alpha q^{x_{j}-i_j-\frac{1}{2}}}{(1+\alpha q^{x_{j}-i_j-\frac{1}{2}})(1+\alpha q^{x_{j}-i_j+\frac{1}{2}})}.
\end{equation}
\end{theorem}

\begin{proof}
By Theorem \ref{thm:MPM-shift} and Remark \ref{rmk:inverse}, we have 
\begin{multline}
\mathbb{P} \left(  \omega(i_1) =x_{1}, \omega(i_2)=x_{2}, \cdots,  \omega(i_k)=x_{k} \right) \\
=\mathbb{P} \left(  \omega(1) =x_{1}-i_1+1, \omega(2)=x_{2}-i_2+2, \cdots,  \omega(k)=x_{k}-i_k+k  \right) \\
=\prod_{j=1}^{k}\frac{(1-q) \alpha q^{x_{j}-i_j-\frac{1}{2}}}{(1+\alpha q^{x_{j}-i_j-\frac{1}{2}})(1+\alpha q^{x_{j}-i_j+\frac{1}{2}})}.
\end{multline}
\end{proof}

\begin{remark}
\label{rmk:MPM-displacements}
For any $k$ integers $i_1<i_2<\cdots<i_k$ and $d_1+i_1>d_2+i_2>\cdots>d_k+i_k$, the formula above in terms of displacements can be written as 
\begin{equation}
\mathbb{P}\left(D_{i_1}=d_1, D_{i_2}=d_2, \cdots, D_{i_k}=d_k\right) = \prod_{j=1}^{k}\frac{(1-q)\alpha q^{d_{j}-\frac{1}{2}}}{(1+\alpha q^{d_{j}-\frac{1}{2}})(1+\alpha q^{d_{j}+\frac{1}{2}})}.
\end{equation}
\end{remark}

\begin{remark}[Conditional independence]
By \eqref{eq:single-i}, we have 
\begin{align}
\label{eq:single-i-2}
\mathbb{P} \left( \omega(i_j)=x_{j} \right) =\mathbb{P} \left( \omega(i_j)-i_j=x_{j}-i_j \right) = \frac{(1-q) \alpha q^{x_{j}-i_j -1/2}}{ (1+ \alpha q^{x_{j}-i_j -1/2}) (1+ \alpha q^{x_{j}-i_j +1/2})},
\end{align}
Therefore, \eqref{eq:separate} implies
\begin{align}
\mathbb{P} \left(  \omega(i_1) =x_{1}, \omega(i_2)=x_{2}, \cdots,  \omega(i_k)=x_{k} \right) =
\prod_{j=1}^{k}\mathbb{P} \left( \omega(i_j) =x_{j} \right),
\end{align}
\textit{in the case} $x_{1}>x_{2}>\cdots>x_{k}$. Therefore, \eqref{eq:separate} can be interpreted as the conditional independence of images of elements of the random permutation distributed according to the Mallows product measure. 
\end{remark}


We also prove the following product formula for a different observable.

\begin{theorem}
\label{prop:correlation}
Let $x_1 \geq x_2 \geq \cdots \geq x_k$, $i_1 < i_2 < \dots < i_k$ be integers, and let $\omega$ be the random permutation distributed according to the Mallows Product Measure $\mathcal{M}^{p}_{\alpha}$. One has
\begin{align}
\label{eq:correlation}
\mathbb{P} \left( \omega(i_1) \leq x_1, \omega(i_2) \leq x_2, \cdots, \omega(i_k) \leq x_k \right) 
=\prod_{j=1}^{k} \frac{1}{1+\alpha q^{x_j-i_j+\frac{1}{2}}}.
\end{align}
\end{theorem}

\begin{proof}
By Theorem \ref{thm:G-shift-invariance} and Remark \ref{rmk:Bernoulli}, we have 
\begin{multline}
\mathbb{P} \left( \omega(i_1) \leq x_1, \omega(i_2) \leq x_2, \cdots, \omega(i_k) \leq x_k \right)\\
=\mathbb{P} \left( \omega(i_1-x_1) \leq 0, \omega(i_2-x_2) \leq 0, \cdots, \omega(i_k-x_k) \leq 0 \right)\\
= \prod_{j=1}^{k} \mathbb{P} \left( \omega(i_j-x_j) \leq 0 \right) 
=\prod_{j=1}^{k} \frac{\alpha^{-1}q^{i_j-x_j-\frac12}}{1+\alpha^{-1}q^{i_j-x_j-\frac12}}=\prod_{j=1}^{k} \frac{1}{1+\alpha q^{x_j-i_j+\frac{1}{2}}}.
\end{multline}
The proof is complete. 
\end{proof}

\begin{remark}
\label{rmk:ext-Ber}
For $x_1=x_2= \dots = x_k=0$, one has
$$
\mathbb{P} \left( \omega(i_1) \leq 0, \omega(i_2) \leq 0, \cdots, \omega(i_k) \leq 0 \right) 
=\prod_{j=1}^{k} \frac{1}{1+\alpha q^{-i_j+\frac{1}{2}}},
$$
which implies that the one-species projection of the Mallows product measure is the Bernoulli product measure of Definition \ref{prop:ASEP blocking}, as expected. 

\end{remark}

\section{Interpretation in terms of particle processes}
\label{sec:particle-process}

Let $\omega$ be the random permutation $\mathbb{Z}\to\mathbb{Z}$ distributed according to the Mallows product measure, and we interpret $\omega(i)=j$ as information that in a particle configuration a particle labeled by $j$ stands at position $i$. Let $\phi: \mathbb{Z}\to\{0,1,2,\cdots,N\}$ be a non-decreasing map, then $\phi(\omega(\cdot)): \mathbb{Z}\to\{0,1,2,\cdots,N\}$ is a random configuration of particles on $\mathbb{Z}$ with labels $0,1,2,\cdots,N$. It follows immediately from q-exchangeability that the distribution of this random particle configuration is reversible under the ASEP dynamics. In our notation, the particles with the largest label $N$ are first-class particles, the particles with label $N-1$ are second class particles, etc. We also say the particles with label $0$ are holes. The map $\phi$, which projects infinite permutation systems to finitely many classes, allows to transfer results about the Mallows product measure to product blocking measures of ASEP, as well as of all other interacting particle systems generated by random walks on Hecke algebras (see \cite{Buf20}).

\subsection{ASEP with $d$ second class particles}

Let $\phi$ be of the following form:
\begin{align}
\phi(i)=
\begin{cases}
0, & i \leq 0 \\
1, & i \in \{1,2,\cdots, d\} \\
2, & i > d 
\end{cases}
\end{align}
Then $\phi(\omega (\cdot))$, where $\omega$ is distributed according to the Mallows product measure, coincides with the product blocking measure of ASEP with $d$ second class particles. We denote the position of the $i$-th second class particle by $s_i$, $i=1,2,\cdots,d$. Namely, we have $\phi(\omega ( s_i))=1$ for $i=1,2,\cdots, d$, $s_1 < s_2 < \dots < s_d$.  

\begin{theorem}
\label{thm:d-second}

Let $x_1 < x_2 < \cdots < x_d$ be arbitrary integers. One has
\begin{align}
\label{eq:d-second}
\mathbb{P}(s_1 = x_1, s_2 = x_2, \dots, s_d = x_d)=\frac{\alpha^{d}q^{\sum_{i=1}^{d}x_{i}-\frac{d(2d+1)}{2}}\prod_{i=1}^{d}(1-q^i)}{\prod_{j=1}^{d}(1+\alpha q^{x_{j}+j-d-\frac{3}{2}})(1+\alpha q^{x_{j}+j-d-\frac{1}{2}})}.
\end{align}
\end{theorem}

\begin{proof}
By the definition, we have 
\begin{equation*}
\mathbb{P}(s_1 = x_1, s_2 = x_2, \dots, s_d = x_d)
=\sum_{\sigma \in S_{d}} \mathbb{P} \left( \omega(x_1)=\sigma(1), \omega(x_2)=\sigma(2), \cdots, \omega(x_d)=\sigma(d) \right)
\end{equation*}
By the invariance under the inversion property, this equality can be continued as  
\begin{multline*}
\ldots =\sum_{\sigma \in S_{d}} \mathbb{P} \left( \omega^{-1}(x_1)=\sigma(1), \omega^{-1}(x_2)=\sigma(2), \cdots, \omega^{-1}(x_d)=\sigma(d) \right)\\
=\sum_{\sigma \in S_{d}} \mathbb{P} \left( \omega (\sigma(1))=x_1, \omega (\sigma(2))=x_2, \cdots, \omega (\sigma(d))=x_d \right)\\
=\sum_{\sigma \in S_{d}} \mathbb{P} \left( \omega (1)=x_{\sigma^{-1}(1)}, \omega (2)=x_{\sigma^{-1}(2)}, \cdots, \omega(d)=x_{\sigma^{-1}(d)} \right)
\end{multline*}
Applying the q-exchangeability property, we obtain
\begin{multline*}
\mathbb{P}(s_1 = x_1, s_2 = x_2, \dots, s_d = x_d)
=\sum_{\sigma \in S_{d}}q^{inv(\sigma^{-1})} \mathbb{P} \left(\omega(1)=x_1, \omega(2) =x_2, \cdots, \omega(d) =x_d \right)\\
=\prod_{i=1}^{d}\frac{1-q^i}{1-q} \mathbb{P} \left(\omega(1)-1=x_1-1, \omega(2)-2 =x_2-2, \cdots, \omega(d)-d =x_d-d \right)\\
\end{multline*}
Applying Theorem \ref{prop:neighbor}, we get \eqref{eq:d-second}.
\end{proof}

\begin{remark}
Theorem \ref{thm:d-second} recovers the recent result of \cite[Theorem 1.6]{ABJ23} by matching the notations $\alpha=q^{\frac12-c}$.
\end{remark}

\subsection{ASEP with one second class, one third class, $\cdots$, one $(d+1)$-st class particles}

Let $\phi$ be of the following form:
\begin{align}
\phi(i)=
\begin{cases}
0, & i \leq 0 \\
i, & i \in \{1,2,\cdots, d\} \\
d+1, & i > d 
\end{cases}
\end{align}
This projection corresponds to the product blocking measure of ASEP with one second class, one third class, $\cdots$, one $(d+1)$-st class particles. We denote the position of the $i$-th class particle by $s_i$, $i=2,3,\cdots,d+1$. Namely, we have $\omega(s_i)=d-i+2$ for $i=2,3,\cdots,d+1$, where $\omega$ is the random permutation distributed according to the Mallows product measure.
\begin{theorem}
\label{thm:multi-class}
For any distinct fixed $x_1,x_2,\cdots,x_d \in \mathbb{Z}$, the probability of the second class, the third class, $\cdots$, the $(d+1)$-st class particles staying at positions $x_1,x_2,\cdots,x_d$ is given by the following formula:
\begin{multline}
\label{eq:multi-class}
\mathbb{P} \left(s_2 =x_1, s_3 =x_2, \cdots, s_{d+1}=x_d  \right)\\
=q^{inv(\sigma)}\frac{(1-q)^{d}\alpha^{d}q^{\sum_{i=1}^{d}x_{i}-\frac{d(2d+1)}{2}}}{\prod_{j=1}^{d}(1+\alpha q^{x_{\sigma(d+1-j)}+j-d-\frac{3}{2}})(1+\alpha q^{x_{\sigma(d+1-j)}+j-d-\frac{1}{2}})},
\end{multline}
where $inv(\sigma)$ is the number of inversions in the word $(x_1, x_2, \dots, x_d)$.
\end{theorem}

\begin{proof}
By the definition, we have 
\begin{align*}
\mathbb{P} \left(s_2 =x_1, s_3 =x_2, \cdots, s_{d+1}=x_d  \right)
=\mathbb{P} \left(\omega (x_1) =d, \omega(x_2) =d-1, \cdots, \omega (x_d) =1\right)
\end{align*}
By the inverse-invariant property, we have 
\begin{multline*}
\mathbb{P} \left(s_2 =x_1, s_3 =x_2, \cdots, s_{d+1}=x_d  \right)
=\mathbb{P} \left(\omega^{-1}(x_1) =d, \omega^{-1}(x_2) =d-1, \cdots, \omega^{-1}(x_d) =1\right)\\
= \mathbb{P} \left(\omega(1) =x_d, \omega(2) =x_{d-1}, \cdots, \omega(d) =x_1\right)
\end{multline*}
By the q-exchangeability property, we have 
\begin{multline*}
\mathbb{P} \left(s_2 =x_1, s_3 =x_2, \cdots, s_{d+1}=x_d  \right)\\
=q^{inv(\sigma)} \mathbb{P} \left(\omega(1) =x_{\sigma(d)}, \omega(2) =x_{\sigma(d-1)}, \cdots, \omega(d) =x_{\sigma(1)} \right)
\end{multline*}
where $\sigma$ is the permutation such that $x_{\sigma(d)} <x_{\sigma(d-1)}< \cdots < x_{\sigma(1)}$. By Theorem \ref{prop:neighbor}, we get \eqref{eq:multi-class}.
\end{proof}

\begin{remark}
It would be interesting to establish the connection of the Mallows product measure and its properties with dualities and the structure of shocks of reversible measures for $n$-species ASEP found by Belitsky-Schütz \cite{BS18}. We do not address this question in the paper. 

\end{remark}

\subsection{Dynamics of a second class particle in ASEP blocking measure}

In this section, we consider the projection map:
\begin{align}
\phi(i)=
\begin{cases}
0, & i < 0 \\
1, & i = 0 \\
2, & i > 0
\end{cases}
\end{align}
We get the ASEP product blocking measure with one second class particle. The distribution of this second class particle is given by \eqref{eq:single}. Here we are interested in the ASEP dynamics of the second class particle which starts with the stationary ASEP product blocking measure. Denote the position of the second class particles at time $t \in \mathbb{R}_{\geq 0}$ by $x(t)$.
Our goal is to compute the transition rates of the form 
$$
\mbox{rate}(x \to y):= \lim_{\Delta t \to 0} \frac{\mathbb{P} \left( x(t+\Delta t) = y | x(t) = x \right)}{ \Delta t }, \qquad x \ne y.
$$
Due to the stationarity, it is sufficient to compute it for $t=0$. 

The following result computes these rates:
\begin{theorem}
\label{thm:dynamic-second}
One has
\begin{align}
\mbox{rate}(x \to y)=
\begin{cases}
\frac{q^{-2x}(1+\alpha q^{x-\frac12})(1+\alpha q^{x+\frac12})}{(1+\alpha q^{-x+\frac12})(1+\alpha q^{-x-\frac32})}, & y=x+1 \\
\frac{q^{-2x+1}(1+\alpha q^{x-\frac12})(1+\alpha q^{x+\frac12})}{(1+\alpha q^{-x-\frac12})(1+\alpha q^{-x+\frac32})}, & y=x-1 \\
0, & \mbox{otherwise}.
\end{cases}
\end{align}
\end{theorem}

\begin{proof}
Let us notice that 
\begin{align*}
\mathbb{P} (x(t)=x\ \mbox{and position} \ x+1 \ \mbox{is a hole at time} \ t)= \mathbb{P}(\omega(x)=0,\omega(x+1)<0),
\end{align*}
where $\omega$ is distributed according to the Mallows product measure. By the translation invariance of the Mallows product measure,
\begin{align*}
\mathbb{P} (x(t)=x \ \mbox{and position} \ x+1 \ \mbox{is a hole at time} \ t)=\mathbb{P} (\omega(0)=-x,\omega(1)<-x),
\end{align*}
Using the result of Theorem \ref{prop:neighbor}, we compute
\begin{multline*}
\mathbb{P} (x(t)=x\ \mbox{and position} \ x+1 \ \mbox{is a hole at time} \ t)\\
=\sum_{y=-\infty}^{-x-1} \mathbb{P} (\omega(0)=-x,\omega(1)=y)=q\sum_{y=-\infty}^{-x-1} \mathbb{P} (\omega(0)=y,\omega(1)=-x)\\
=q\sum_{y=-\infty}^{-x-1} \mathbb{P} (D_0=y, D_1=-x-1)=\frac{(1-q)\alpha q^{-x-\frac12}}{(1+\alpha q^{-x+\frac12})(1+\alpha q^{-x-\frac12})(1+\alpha q^{-x-\frac32})},
\end{multline*}
where we use the q-exchangeability of the Mallows product measure and the calculation:
\begin{align*}
&\sum_{y=-\infty}^{-x-1} \mathbb{P} (D_0=y, D_1=-x-1)\\
=& \sum_{y=-\infty}^{-x-1}\frac{(1-q)^2 \alpha^2 q^{y-x-3}}{(1+\alpha q^{y-\frac32})(1+\alpha q^{y-\frac12})(1+\alpha q^{-x-\frac12})(1+\alpha q^{-x+\frac12})}\\
=& \frac{(1-q) \alpha q^{y-\frac32}}{(1+\alpha q^{-x-\frac12})(1+\alpha q^{-x+\frac12})}\sum_{y=-\infty}^{-x-1}\frac{(1-q) \alpha q^{y-\frac32}}{(1+\alpha q^{y-\frac32})(1+\alpha q^{y-\frac12})}\\
=& \frac{(1-q) \alpha q^{y-\frac32}}{(1+\alpha q^{-x-\frac12})(1+\alpha q^{-x+\frac12})}\sum_{y=-\infty}^{-x-1}\left(\frac{1}{1+\alpha q^{y-\frac12}}-\frac{1}{1+\alpha q^{y-\frac32}}\right)\\
=& \frac{(1-q)\alpha q^{-x-\frac32}}{(1+\alpha q^{-x+\frac12})(1+\alpha q^{-x-\frac12})(1+\alpha q^{-x-\frac32})}.
\end{align*}
Similarly, we have 
\begin{multline*}
\mathbb{P} (x(t)=x\ \mbox{and position} \ x+1 \ \mbox{contains a first class particle at time} \ t)\\
=\mathbb{P} (\omega(x)=0,\omega(x+1)>0)= \mathbb{P} (\omega(0)=-x,\omega(1)>-x)\\
=\sum_{y=-x+1}^{+\infty} \mathbb{P} (\omega(0)=-x,\omega(1)=y)=\frac{(1-q)\alpha^2 q^{-2x-2}}{(1+\alpha q^{-x+\frac12})(1+\alpha q^{-x-\frac12})(1+\alpha q^{-x-\frac32})}.
\end{multline*}
Next, by the ASEP dynamics rules, we have 
\begin{multline*}
\mathbb{P}(x(t)=x, x(t+\Delta t)=x+1)\\
=1 \cdot \mathbb{P}(x(t)=x, \mbox{position} \ x+1 \ \mbox{is vacant at time} \ t) \cdot \Delta t\\
+q \cdot \mathbb{P} (x(t)=x, \mbox{position} \ x+1 \ \mbox{is occupied by a first class particle at time} \ t) \cdot \Delta t \\
+o(\Delta t)=\Delta t \cdot \frac{(1-q)\alpha q^{-x-\frac12}}{(1+\alpha q^{-x+\frac12})(1+\alpha q^{-x-\frac32})}+o(\Delta t).
\end{multline*}
Finally, plugging the obtained formulas into 
\begin{align*}
\mbox{rate}(x \to x+1) = \lim_{\Delta t \to 0}\frac{\mathbb{P} (x(t)=x, x(t+\Delta t)=x+1)}{\mathbb{P}(x(t)=x)},
\end{align*}
we obtain the expression for the $\mbox{rate}(x \to x+1)$ in the proposition. An analogous calculation gives 
\begin{align*}
\mathbb{P} (x(t)=x, x(t+\Delta t)=x-1)=\Delta t \cdot \frac{(1-q)\alpha q^{-x+\frac12}}{(1+\alpha q^{-x-\frac12})(1+\alpha q^{-x+\frac32})}+o(\Delta t).
\end{align*}
which leads to the formula for $\mbox{rate}(x \to x-1)$.
\end{proof}

\subsection{Projection to ASEP(q,M)}
\label{sec:ASEP-qm}

We start with recalling the definition of the process ASEP($q,M$), with $M \in \N$. Each position from $\mathbb{Z}$ can be occupied from 0 to $M$ particles. The single-species version of the process was introduced in \cite{CGRS16}, and the multi-species version in \cite{K18}, \cite{K19}. We will consider configurations with infinitely many first-class particles and exactly one second class particle. Below we denote the content of a position by $(a,b)$, meaning that this position contains $a$ first-class particles, and $b$ second class particles ($a+b \le M$, $b \in \{ 0,1 \}$). 

The process runs in continuous time, and particles may jump one step to the right or to the left; the transition rates of such jumps depend on the amount of particles in the two involved neighboring positions.
We provide below all transition rates of ASEP(q,M) with one second class particle. They can be obtained from \cite{K18}, \cite{K19}, or from an equivalent algebraic construction \cite[Section 3.5]{Buf20}. For two neighboring positions one particle is allowed to jump from one to another, and also the first-class particle can push back the second-class particle; the exact transition rates are:  

\begin{align*}
 \{ (n_1, 1), (n_2,0) \} \to  \{ (n_1-1, 1), (n_2+1,0) \}  = \frac{(1-q^{n_1})(1-q^{M-n_2})}{(1-q^M)^2} \\
 \{ (n_1, 1), (n_2,0) \} \to  \{ (n_1, 0), (n_2,1) \}  = \frac{(1-q) q^{n_1} (1-q^{M-n_2})}{(1-q^M)^2} \\
 \{ (n_1, 1), (n_2,0) \} \to  \{ (n_1+1, 1), (n_2-1,0) \}  = \frac{ ( q^{n_1+1} - q^M) (q^{M-n_2}-q^M) q }{(1-q^M)^2} \\
 \{ (n_1, 1), (n_2,0) \} \to  \{ (n_1+1, 0), (n_2-1,1) \}  = \frac{(1-q) q^{n_1} (q^{M-n_2} - q^{M}) q}{(1-q^M)^2} \\
 \end{align*}
 \begin{align*}
 \{ (n_1, 0), (n_2,1) \} \to  \{ (n_1-1, 0), (n_2+1,1) \}  = \frac{(1-q^{n_1}) (1-q^{M-n_2-1})}{(1-q^M)^2} \\
 \{ (n_1, 0), (n_2,1) \} \to  \{ (n_1-1, 1), (n_2+1,0) \}  = \frac{(1-q^{n_1})(1-q) q^{M-n_2} }{(1-q^M)^2} \\
 \{ (n_1, 0), (n_2,1) \} \to  \{ (n_1+1, 0), (n_2-1,1) \}  = \frac{(q^{n_1} -q^M) (q^{M-n_2} - q^M) q}{(1-q^M)^2} \\
 \{ (n_1, 0), (n_2,1) \} \to  \{ (n_1, 1), (n_2,0) \}  = \frac{(q^{n_1} -q^M) (1-q) q^{M-n_2+1}}{(1-q^M)^2} \\
\end{align*}

As shown in \cite{Buf20}, ASEP($q,M$) can be realized as a random walk on Hecke algebra. This implies that the Mallows Product measure is also the product blocking (=reversible) measure of ASEP($q,M$). Therefore, the distribution of the single second class particle under this measure can be obtained in the following way.

We consider the projection map
\begin{align}
\phi(i)=
\begin{cases}
1, & i < 0 \\
2, & i =0 \\
3, & i > 0 
\end{cases}
\end{align}
This leads to ASEP($q,M$) with exactly one second class particle on the whole line. Let us denote the position of the second-class particle by $s_2$. To each position $x \in \Z$ of ASEP($q,M$) corresponds the interval of integers $[xM+1; (x+1)M]$, see \cite{Buf20}. Therefore, the distribution of the second-class particle is given by the following formula, where $\omega$ is the random permutation distributed according to the Mallows product measure: 
\begin{align}
\mathbb{P} \left( s_2 =x \right)=\sum_{i=1}^{M} \mathbb{P} \left( \omega(0) =x M +i \right)=\frac{(1-q^M)\alpha q^{xM+\frac12}}{(1+\alpha q^{(x+1)M+\frac12})(1+\alpha q^{xM+\frac12})}.
\end{align}

\section{Asymptotics}
\label{sec:asymptotics}

In this short section we record limiting expressions for results of the last two sections. We scale the parameter $q$ and the position variable $x$ as 
\begin{align}
\label{eq:scalings}
q=e^{-\epsilon}, \quad x= \lfloor y/\epsilon \rfloor, \quad y \in \mathbb{R}, \quad \epsilon \to 0.
\end{align}

\begin{proposition}
\label{prop:scaling-neighbor}
Let $\omega$ be the random permutation distributed according to the Mallows product measure $\mathcal{M}^{p}_{\alpha}$. For fixed $k \in \mathbb{N}$, under the scaling \eqref{eq:scalings}, we have
\begin{align}
\label{eq:scaling-neighbor}
\left(\epsilon \omega(0), \epsilon \omega(1),\cdots,\epsilon \omega(k-1) \right) \underset{\epsilon \rightarrow 0}{\Longrightarrow} (y_1,y_2,\cdots,y_k).
\end{align}
where $y_1, y_2, \cdots, y_k$ are i.i.d. random variables, and $y_i$ follows the logistic distribution with density:
\begin{align*}
\frac{\alpha e^{-t}}{(1+\alpha e^{-t})^2} dt.
\end{align*}
\end{proposition} 

\begin{proof}
Applying Theorem \ref{prop:neighbor}, one has
\begin{multline}
\label{eq:sacle-neighbor1}
\mathbb{P} \left( \omega(0) =\lfloor y_1 /\epsilon \rfloor, \omega(1)-1 =\lfloor y_2 /\epsilon \rfloor, \cdots, \omega(k-1)-(k-1) =\lfloor y_k /\epsilon \rfloor \right) \\ \approx \frac{(1-e^{-\epsilon})^{k}\alpha^{k}e^{-\sum_{i=1}^{k}y_i-\frac{k^2}{2}\epsilon}}{\prod_{j=1}^{k}(1+\alpha e^{-\epsilon(\frac{y_j}{\epsilon}+2j-k-\frac{3}{2})})(1+\alpha e^{-\epsilon(\frac{y_j}{\epsilon}+2j-k-\frac{1}{2})})}
\approx \epsilon^{k}\frac{\alpha^k e^{-\sum_{i=1}^{k}y_i}}{\prod_{j=1}^{k}(1+\alpha e^{-y_j})^2}.
\end{multline}
This implies Proposition \ref{prop:scaling-neighbor}.
\end{proof}

\begin{proposition}
\label{prop:scaling-correlation}

Let $\omega$ be the random permutation distributed according to the Mallows product measure $\mathcal{M}^{p}_{\alpha}$. For fixed $k \in \mathbb{N}$, under the scaling \eqref{eq:scalings}, we have
\begin{align}
\label{eq:scaling-correlation}
\mathbb{P} \left( \omega(0) \leq \frac{y_1}{\epsilon}, \omega(1) \leq \frac{y_2}{\epsilon}, \cdots, \omega(k-1) \leq \frac{y_k}{\epsilon} \right) 
 \xrightarrow[\epsilon \to 0]{} \frac{1}{\prod_{j=1}^{k}(1+\alpha e^{-y_j})}.
\end{align}
\end{proposition}

\begin{proof}
By the result of Theorem \ref{prop:correlation}, for $y_1 \ge y_2 \ge \dots \ge y_k$ one has
\begin{multline}
\mathbb{P} \left( \omega(0) \leq \frac{y_1}{\epsilon}, \omega(1) \leq \frac{y_2}{\epsilon}, \cdots, \omega(k-1) \leq \frac{y_k}{\epsilon} \right)  \\
=\frac{1}{\prod_{j=1}^{k}(1+\alpha e^{-y_j-(\frac32-j)\epsilon})}\approx \frac{1}{\prod_{j=1}^{k}(1+\alpha e^{-y_j})}.
\end{multline}
Combined with q-exchangeability, this implies Proposition \ref{prop:scaling-correlation}. Alternatively, this Proposition immediately follows from Proposition \ref{prop:scaling-neighbor} (and vice versa).
\end{proof}

The calculation for the next Proposition is also immediate. 
\begin{proposition}
\label{prop:scaling-second}
The rates from Theorem \ref{thm:dynamic-second} have the following limit under the scaling \eqref{eq:scalings}:
\begin{align}
\lim_{\epsilon \to 0} \left( rate(x \to x+1) \right) = \lim_{\epsilon \to 0} \left( rate(x \to x-1) \right) = \frac{ e^{2y}(1+\alpha e^{-y})^2}{(1+\alpha e^y)^2}.
\end{align}

\end{proposition}

\begin{remark}
It is possible to write down formulas for the ergodic Mallows measures with the help of the Mallows product measure. For this, one reverses the mixing in \eqref{def:Product Mallows} via integrating over parameter $\alpha$ (see \cite{Bor07}). In particular, this would allow to prove very similar asymptotics for the ergodic Mallows measures. However, we do not address this question in the current paper. 

\end{remark}

\appendix 

\section{Proofs of some identities}

\begin{lemma}
\label{lem:A.1}
Let $\alpha \in \mathbb{R}_{>0}$, $q\in[0,1)$ and $x \in \mathbb{Z}$, we have the following identity:
\begin{multline}
\label{eq:A.1}
\frac{ (1-q) \alpha^{x} q^{x^2/2} \prod_{m=0}^{\infty} (1 + \alpha^{-1} q^{-x+3/2 +m}) \prod_{n=0}^{\infty} (1 + \alpha q^{x+3/2 +n}) }{\prod_{r=0}^{\infty} \left( 1+\alpha q^{r+1/2} \right) \prod_{\ell=0}^{\infty} \left( 1+\alpha^{-1} q^{\ell+1/2} \right)}
= \frac{(1-q) \alpha q^{x-1/2}}{ (1+ \alpha q^{x-1/2}) (1+ \alpha q^{x+1/2})}.
\end{multline}
\end{lemma}

\begin{proof}
We prove \eqref{eq:A.1} by checking case by case:
\begin{itemize}
\item When $x>1$, note that
\[\prod_{m=0}^{\infty} (1 + \alpha^{-1} q^{-x+3/2 +m})=\prod_{m=0}^{\infty} (1 + \alpha^{-1} q^{1/2 +m}) \cdot \prod_{m=0}^{x-2}(1 + \alpha^{-1} q^{-x+3/2 +m}),\]
\[\prod_{n=0}^{\infty} (1 + \alpha q^{x+3/2 +n}) \cdot \prod_{n=0}^{x} (1 + \alpha q^{1/2 +n})=\prod_{n=0}^{\infty} \left( 1+\alpha q^{n+1/2} \right),\]
\[\prod_{m=0}^{x-2}(1 + \alpha^{-1} q^{-x+3/2 +m})=\alpha^{-(x-1)} q^{-(x-1)^2/2}\prod_{m=0}^{x-2}(1 + \alpha q^{x-3/2-m}),\]
then we have
\[LHS=\frac{(1-q) \alpha q^{x-1/2} \prod_{m=0}^{x-2}(1 + \alpha q^{x-3/2-m})}{ \prod_{n=0}^{x} (1 + \alpha q^{1/2 +n})}=RHS.\]

\item When $x=1$, note that
\[\prod_{m=0}^{\infty} (1 + \alpha^{-1} q^{-x+3/2 +m})=\prod_{m=0}^{\infty} (1 + \alpha^{-1} q^{1/2 +m}),\] 
\[\prod_{n=0}^{\infty} (1 + \alpha q^{x+3/2 +n})(1+ \alpha q^{1/2}) (1+ \alpha q^{3/2})=\prod_{n=0}^{\infty} \left( 1+\alpha q^{n+1/2} \right),\]
then we have
\[LHS=\frac{(1-q) \alpha q^{1/2}}{ (1+ \alpha q^{1/2}) (1+ \alpha q^{3/2})}=RHS.\]

\item When $x=0$, note that
\[\prod_{m=0}^{\infty} (1 + \alpha^{-1} q^{-x+3/2 +m})(1 + \alpha^{-1} q^{1/2})=\prod_{m=0}^{\infty} \left( 1+\alpha^{-1} q^{m+1/2} \right),\]
\[\prod_{n=0}^{\infty} (1 + \alpha q^{x+3/2 +n}) (1 + \alpha q^{1/2})=\prod_{n=0}^{\infty} (1 + \alpha q^{1/2 +n}),  \]
then we have 
\[LHS=\frac{1-q}{ (1 + \alpha^{-1} q^{1/2})(1 + \alpha q^{1/2})}=\frac{(1-q) \alpha q^{-1/2}}{ (1 + \alpha q^{-1/2})(1 + q^{1/2})}=RHS.\]

\item When $x=-1$, note that 
\[\prod_{m=0}^{\infty} (1 + \alpha^{-1} q^{-x+3/2 +m}) (1 + \alpha^{-1} q^{3/2})(1 + \alpha^{-1} q^{1/2})=\prod_{m=0}^{\infty} \left( 1+\alpha^{-1} q^{m+1/2} \right),\]
\[\prod_{n=0}^{\infty} (1 + \alpha q^{x+3/2 +n})=\prod_{n=0}^{\infty} (1 + \alpha q^{1/2 +n}),  \]
then we have 
\[LHS=\frac{(1-q) \alpha^{-1} q^{1/2}}{ (1 + \alpha^{-1} q^{3/2})(1 + \alpha^{-1} q^{1/2})}=\frac{(1-q) \alpha q^{-3/2}}{ (1 + \alpha q^{-3/2})(1 + q^{-1/2})}=RHS.\]

\item When $x<-1$, note that 
\[\prod_{m=0}^{\infty} (1 + \alpha^{-1} q^{-x+3/2 +m}) \cdot \prod_{m=0}^{-x} (1 + \alpha^{-1} q^{1/2 +m})=\prod_{m=0}^{\infty} (1 + \alpha^{-1} q^{1/2 +m}),\]
\[\prod_{n=0}^{\infty} (1 + \alpha q^{x+3/2 +n})=\prod_{n=0}^{\infty} \left( 1+\alpha q^{n+1/2} \right) \cdot \prod_{n=0}^{-x-2} (1 + \alpha q^{x+3/2 +n}),\]
\[\prod_{m=0}^{-x} (1 + \alpha^{-1} q^{1/2 +m})=\alpha^{x-1} q^{(x-1)^2/2}\prod_{m=0}^{-x} (1 + \alpha q^{-1/2-m}),\]
then we have
\[LHS=\frac{(1-q) \alpha q^{x-1/2} \prod_{n=0}^{-x-2} (1 + \alpha q^{x+3/2 +n})}{\prod_{m=0}^{-x} (1 + \alpha q^{-1/2-m})}=RHS.\]
\end{itemize}
\end{proof}

\begin{lemma}
\label{lem:A.2}
Let $\alpha \in \mathbb{R}_{>0}$, $q\in[0,1)$, $k\in\mathbb{Z}_{+}$ and $x_1 \leq x_2 \leq \cdots \leq x_k$ be integers, we have the following identity:
\begin{multline}
\label{eq:A.2}
(1-q)^{k} \alpha^{x_1} q^{\frac{x_1^2}{2}+x_2+\cdots+x_k-(k-1)x_1}
\prod_{i=1}^{k-1}\prod_{j=0}^{x_{i+1}-x_i-1}\left(1+\alpha q^{j+x_i+\frac{1}{2}+(2i-k)}\right)\\
\times \frac{\prod_{m=0}^{\infty}\left(1+\alpha q^{m+x_k+\frac{1}{2}+k}\right)
\prod_{n=0}^{\infty}\left(1+\alpha^{-1} q^{n-x_1+\frac{1}{2}+k}\right)}{\prod_{r=0}^{\infty} \left( 1+\alpha q^{r+1/2} \right) \prod_{\ell=0}^{\infty} \left( 1+\alpha^{-1} q^{\ell+1/2} \right)}=\frac{(1-q)^{k}\alpha^{k}q^{\sum_{j=1}^{k}x_{j}-\frac{k^2}{2}}}{\prod_{j=1}^{k}(1+\alpha q^{x_j+2j-k-\frac{3}{2}})(1+\alpha q^{x_j+2j-k-\frac{1}{2}})}.
\end{multline}
\end{lemma}

\begin{proof}
To prove \eqref{eq:A.2}, we only need to prove the following identity:
\begin{align*}
\alpha^{x_1-k} q^{\frac{(x_1-k)^2}{2}}\frac{\prod_{m=0}^{\infty}\left(1+\alpha q^{m+x_k+\frac{1}{2}+k}\right)
\prod_{n=0}^{\infty}\left(1+\alpha^{-1} q^{n-x_1+\frac{1}{2}+k}\right)}{\prod_{r=0}^{\infty} \left( 1+\alpha q^{r+1/2} \right) \prod_{\ell=0}^{\infty} \left( 1+\alpha^{-1} q^{\ell+1/2} \right)}=\frac{1}{\prod_{m=x_1-k}^{x_k+k-1} \left( 1+\alpha q^{m+1/2} \right)}.
\end{align*}
The proof of above identity is similar with the proof of \eqref{eq:A.1}, we should check case by case: $x_k \geq x_1 >k$, $x_k \geq x_1 =k$, $k>x_k \geq x_1 >-k$, $x_k=-k \geq x_1 $ and $-k>x_k \geq x_1 $. For the sake of brevity, we only provide the proof of the case $x_k \geq x_1 >k$.

When $x_k \geq x_1 >k$, we note that
\[\prod_{m=0}^{\infty} (1 + \alpha q^{m+x_k+\frac{1}{2}+k}) \cdot \prod_{m=0}^{x_k+k-1} (1 + \alpha q^{1/2 +n})=\prod_{m=0}^{\infty} \left( 1+\alpha q^{m+1/2} \right),\]
\[\prod_{n=0}^{\infty} (1 + \alpha^{-1} q^{n-x_1+\frac{1}{2}+k})=\prod_{n=0}^{\infty} (1 + \alpha^{-1} q^{1/2 +n}) \cdot \prod_{n=0}^{x_1-k-1}(1 + \alpha^{-1} q^{n-x_1+\frac{1}{2}+k}),\]
\[\prod_{n=0}^{x_1-k-1}(1 + \alpha^{-1} q^{n-x_1+\frac{1}{2}+k})=\alpha^{-(x_1-k)} q^{\frac{-(x_1-k)^2}{2}}\prod_{n=0}^{x_1-k-1}(1 + \alpha q^{x_1-n-\frac{1}{2}-k}),\]
then we have
\begin{multline*}
\alpha^{x_1-k} q^{\frac{(x_1-k)^2}{2}}\frac{\prod_{m=0}^{\infty}\left(1+\alpha q^{m+x_k+\frac{1}{2}+k}\right)
\prod_{n=0}^{\infty}\left(1+\alpha^{-1} q^{n-x_1+\frac{1}{2}+k}\right)}{\prod_{r=0}^{\infty} \left( 1+\alpha q^{r+1/2} \right) \prod_{\ell=0}^{\infty} \left( 1+\alpha^{-1} q^{\ell+1/2} \right)}\\
=\frac{\prod_{n=0}^{x_1-k-1}(1 + \alpha^{-1} q^{x_1-n-\frac{1}{2}-k})}{\prod_{m=0}^{x_k+k-1} (1 + \alpha q^{1/2 +m})}=\frac{1}{\prod_{m=x_1-k}^{x_k+k-1} \left( 1+\alpha q^{m+1/2} \right)}.
\end{multline*}
\end{proof}

\begin{lemma}
\label{lem:A.3}
Let $i$ be an arbitrary integer, and let $b,k$ be positive integers such that $b<k$. We have the following identity:
\begin{align}
\label{eq:A.3}
\sum_{\ell=0}^{k-b} (-1)^{\ell} \alpha^{-(k-b)} q^{f(\ell)} \frac{(q;q)_{k-b}}{(q;q)_{\ell}(q;q)_{k-b-\ell}} \prod_{j=b+\ell+1}^{k}(1+\alpha q^{x_1+j-i-k-\frac12})=q^{(k-b)x_1+\frac{(k-b)(k-b+1)}{2}}.
\end{align}
where $f(\ell) = \sum_{j=b+1}^{b+\ell}(k-j+i+\frac32)+(k-b-\ell)(k-b-\ell+i+\frac12)$.
\end{lemma}

\begin{proof}
First, we expand the product term in the left-hand side of \eqref{eq:A.3}
\[\prod_{j=b+\ell+1}^{k}(1+\alpha q^{x_1+j-i-k-\frac12}) = \sum_{m=0}^{k-b-\ell} \alpha^m q^{m(x_1-i-k-\frac12)} g_m(\ell),\]
where $g_m(\ell) = q^{m(b+\ell)+\frac{m(m+1)}{2}}  \frac{(q;q)_{k-b-\ell}}{(q;q)_m(q;q)_{k-b-\ell-m}}$. Substituting this into the left-hand side of \eqref{eq:A.3} and rearranging, 
we obtain
\[LHS=\sum_{m=0}^{k-b}  q^{m x_1} \sum_{\ell=0}^{k-b-m} (-1)^{\ell} \alpha^{m-(k-b)} q^{h(\ell, m)} \frac{(q;q)_{k-b}}{(q;q)_{\ell}(q;q)_m(q;q)_{k-b-\ell-m}}, \]
where $h(\ell, m)=f(\ell)-m(i+k+\frac12)+m(b+\ell)+\frac{m(m+1)}{2}$. 
We can rewrite $h(\ell, m)$ as $h(\ell, m)=u(m)+v(\ell, m)$, where $u(m)=(k-b-m)(i+k+\frac12-b)+\frac{m(m+1)}{2}$ and $v(\ell,m)=\frac{\ell(\ell-1)}{2}-(k-b-m-1)\ell$. 
Therefore, we can write the left-hand side of \eqref{eq:A.3} as 
\[LHS=\sum_{m=0}^{k-b}  C(m) q^{m x_1} \sum_{\ell=0}^{k-b-m} (-1)^{\ell} q^{v(\ell, m)} \frac{(q;q)_{k-b-m}}{(q;q)_{\ell}(q;q)_{k-b-m-\ell}},\]
where $C(m)=\alpha^{m-(k-b)} q^{u(m)} \frac{(q;q)_{k-b}}{(q;q)_m (q;q)_{k-b-m}}$. When $m=k-b$, it is easy to see that, 
\begin{multline}
\label{eq:m=k-b}
C(m) q^{m x_1} \sum_{\ell=0}^{k-b-m} (-1)^{\ell} q^{v(\ell, m)} \frac{(q;q)_{k-b-m}}{(q;q)_{\ell}(q;q)_{k-b-m-\ell}}\\
=C(k-b) q^{(k-b) x_1}=q^{u(k-b)} q^{(k-b) x_1}=q^{(k-b)x_1+\frac{(k-b)(k-b+1)}{2}}, \quad m=k-b.
\end{multline}
For $0 \leq m \leq k-b-1$, we use the finite q-Binomial theorem \eqref{eq:q-Binomial} in a modified form:
\[\prod_{r=0}^{n-1}(1-xq^r)=\sum_{p=0}^{n}(-1)^{p}q^{\frac{p(p-1)}{2}}\binom{n}{p}_q x^{p}, \quad \binom{n}{p}_q = \frac{(q;q)_n}{(q;q)_p(q;q)_{n-p}},\]
and, setting $n=k-b-m$, $p=\ell$, and $x=q^{-(k-b-m-1)}$, we obtain
\begin{multline}
\label{eq:m-leq-k-b}
C(m) q^{m x_1} \sum_{\ell=0}^{k-b-m} (-1)^{\ell} q^{v(\ell, m)} \frac{(q;q)_{k-b-m}}{(q;q)_{\ell}(q;q)_{k-b-m-\ell}}\\
=  C(m) q^{m x_1} \prod_{r=0}^{k-b-m-1}(1-q^{r-(k-b-m-1)})=0, \quad 0 \leq m \leq k-b-1. 
\end{multline}
We can get \eqref{eq:A.3} by combining \eqref{eq:m=k-b} and \eqref{eq:m-leq-k-b}.
\end{proof}

\end{document}